\documentclass{amsart}
\usepackage{graphicx,cancel,xcolor,comment}
\def\HH{\mathcal H}
\def\AA{\mathcal A}

\def\DA{D\left(\AA\right)}
\def\Dah{D(A^{\frac{1}{2}})}
\def\Da{D\left(A\right)}

\def\ut{u_t}
\def\utt{u_{tt}}
\def\yt{y_t}
\def\ytt{y_{tt}}

\def\Et{E\left(t\right)}
\newcommand {\nc}   {\newcommand}
\nc {\be}   {\begin{equation}} \nc {\ee}   {\end{equation}} \nc
{\beq}  {\begin{eqnarray}} \nc {\eeq}  {\end{eqnarray}} \nc {\beqs}
{\begin{eqnarray*}} \nc {\eeqs} {\end{eqnarray*}}
\def\edc{\end{document}}

\newtheorem{thm}{Theorem}[section]

\newtheorem{cor}[thm]{Corollary}

\newtheorem{defi}[thm]{Definition}

\newtheorem{lem}[thm]{Lemma}

\newtheorem{prop}[thm]{Proposition}

\numberwithin{equation}{section}

\theoremstyle{definition}

\numberwithin{equation}{section}

 \renewenvironment{proof}{{\bfseries \noindent Proof.}}{\demo}
\newcommand\xqed[1]{%
  \leavevmode\unskip\penalty9999 \hbox{}\nobreak\hfill
  \quad\hbox{#1}}
\newcommand\demo{\xqed{$\square$} \\[0.1in]}

\usepackage{hyperref}

\hypersetup{
   bookmarks,%
                colorlinks,%
                urlcolor=blue,%
                citecolor=blue,%
                linkcolor=blue,%
                hyperfigures,%
                pagebackref,%
                pdfcreator=LaTeX,%
                breaklinks=true,%
                pdfpagelayout=SinglePage,%
                bookmarksopen=true,%
                bookmarksopenlevel=2         
}

\begin{document}
\title[Optimal indirect stability  of a weakly  damped elastic system]{Optimal indirect stability  of a weakly damped elastic abstract system of  second order equations coupled by velocities} 
    \author{Farah Abdallah}
	\address{Lebanese University\\
		Faculty of sciences 1\\ 
		Khawarizmi Laboratory of Mathematics and Applications-KALMA, Hadath-Beirut}
	\email{farahabdallah2@ul.edu.lb}
 \author{Yacine Chitour}
	\address{L2S\\
		Paris-Saclay University, 3 Rue Joliot Curie, Gif-sur-Yvette, France. }
	\email{yacine.chitour@l2s.centralesupelec.fr}
     \author{Mouhammad Ghader}
	\address{Lebanese University\\
		Faculty of sciences 1, Khawarizmi Laboratory of Mathematics and Applications-KALMA, Hadath-Beirut \&  L2S, Paris-Saclay University, 3 Rue Joliot Curie, Gif-sur-Yvette, France. }
	\email{mhammadghader@hotmail.com}

	\author{Ali Wehbe}
	\address{Lebanese University\\
		Faculty of sciences 1\\ 
		Khawarizmi Laboratory of Mathematics and Applications-KALMA, Hadath-Beirut}
	\email{ali.wehbe@ul.edu.lb}
	

\date{}
\begin{abstract}
In this paper, by means of the Riesz basis approach, we study the stability of a weakly damped system of two second order evolution equations coupled through the velocities (see \eqref{Coupled Equation 1.1}). If the fractional order damping becomes viscous and the waves propagate with equal speeds, we prove exponential stability of the system and, otherwise, we establish an optimal  polynomial decay rate. Finally, we provide some illustrative examples.
\end{abstract}
\maketitle
\section{Introduction}\label{Coupled Section 1}
 \noindent In this paper, we investigate   the energy decay rate of the following abstract  system of second order evolution equations
\begin{equation}\label{Coupled Equation 1.1}
\left\{
\begin{array}{lll}

\utt+aAu+A^{\gamma}\ut+\alpha \yt=0,

\\ \noalign{\medskip}

\ytt+Ay-\alpha \ut=0,
\end{array}
\right.
\end{equation}
 where $a>0$, $\gamma\leq0$, $\alpha\in\mathbb{R}^*$ is the coupling parameter and $A$ is a self-adjoint, coercive operator with a compact resolvent in a separable Hilbert space $H$ and with simple spectrum. The fractional  damping term $A^\gamma u_t$ is only applied at the first equation and the second equation is indirectly damped through the coupling between the two equations. The fractional order damping of the type $A^\gamma$, arising from the material property, has been introduced in \cite{Chen01} and, in the cases 
 $\gamma\in\{0,\frac{1}{2},1\}$, is referred to as 
 the so-called viscous damping, square-root (or structural) damping, and Kelvin-Voigt damping respectively. If $\gamma=\frac{1}{2}$, it was shown in \cite{Chen01} that the semigroup corresponding to the damped elastic model
$$\utt+Au+A^{\gamma}\ut=0,$$
is analytical, while the subsequent works in \cite{Chen02} and \cite{Chen03} showed that the semigroup is still analytical for $\frac{1}{2}\leq \gamma \leq 1$ but is only of Gevrey class for {$0<\gamma <\frac{1}{2}$}.
\\[0.1in]

\noindent In \cite{LiuZhang01}, Liu and Zhang   studied the energy decay rate of the weakly damped elastic abstract system described by
\begin{equation}\label{System new-1}
\left\{
\begin{array}{ll}
\displaystyle{u_{tt}+Au+Bu_t=0},\\ \noalign{\medskip}
\displaystyle{u(0)=u_0,\ u_t(0)=u_1},

\end{array}
\right.
\end{equation}
where $A$ is a self-adjoint, positive definite operator on a Hilbert space $H$. The dissipation operator $B$  is another positive operator satisfying $cA^\gamma u \leq Bu \leq CA^\gamma u $ for some constants $0 < c < C.$ When $\gamma < 0,$   they proved that the  energy of System \eqref{System new-1} has a  polynomial  decay rate of type $t^{\frac{1}{\gamma}}$  and that this decay rate is in some sense optimal. Regarding System \eqref{Coupled Equation 1.1} when $\alpha=0$, it  reduces to System \eqref{System new-1} with $B=A^\gamma$ and $a=1$.  In this case, we recover  the results of \cite{LiuZhang01}. \\[0.1in]
\noindent When the coupling acts through displacements, Loreti and Rao studied  in \cite{RaoLoreti01} the stability of  the following abstract system of coupled  equations
\begin{equation}\label{Coupled Equation 1.3}
\left\{
\begin{array}{lll}
\utt+Au+A^{\gamma}\ut+\alpha y=0,
          \\ \noalign{\medskip}
\ytt+Ay-\alpha u=0.
\end{array}\right.
\end{equation}
They proved that System \eqref{Coupled Equation 1.3} is {not  exponentially  stable} and an optimal polynomial energy decay rate of type $t^{-\tau\left(\gamma\right)}$ is obtained where
\begin{equation*}
\tau\left(\gamma\right)=
\left\{
\begin{array}{lll}

\frac{1}{\gamma+1},&-\frac{1}{2}\leq \gamma\leq 0,
         
                  \\  \noalign{\medskip}
         
-\frac{1}{\gamma},&-\frac{1}{2}\geq\gamma.

\end{array}
\right.
\end{equation*} 
Consequently,  the energy achieves its maximum optimal decay rate $t^{-2}$ when $\gamma=-\frac{1}{2}$. System \eqref{Coupled Equation 1.3} is the closest to our System \eqref{Coupled Equation 1.1}. However,  the coupling in \cite{RaoLoreti01}  acts through displacements while in this paper the coupling acts through the velocities. Indeed, the transmission between the two equations depends on the nature of the coupling; for instance, see \cite{Alabau04}, \cite{Alabau02}, \cite{KhodjaBader01}, and \cite{Kapitonov01}. 
\\[0.1in]

 \noindent The fact that only one equation of the coupled system is damped refers to the so-called class of indirect stabilization problems. The concept of indirect damping mechanisms has been introduced by Russell in \cite{Russell01}. That paper  is one of the firsts  to give an algebraic characterization of coupled indirectly damped vibration models.
 Before we start our study, we recall some results concerning the stability of systems of two equations coupled by velocities. In  \cite{Kapitonov01}, Kapitonov  considers a pair of coupled hyperbolic systems in some open subset of a domain.  One of these systems contains locally distributed damping. Under certain conditions imposed on the subset where the damping terms is effective, a uniform decay of the energy is established.  The results are proved by using multiplier techniques. Khodja and Bader in \cite{KhodjaBader01} study the  stability of a system of coupled one-dimensional wave equations posed on a finite interval $(0,1)$ with only one internal or boundary control. They show that the internal damping applied to only one of the equations never gives exponential stability if the wave speeds are different. If the wave speeds are the same, they present necessary and sufficient conditions for stability.  In addition, the simultaneous boundary stabilization of the same system is also studied in \cite{KhodjaBader01}. Let us mention the additional references \cite{Khodja01},  \cite{ Ammari02}, \cite{Wehbe05}, \cite{Ammari01}, and  \cite{Alabau07} for indirect stabilization of coupled equations via  one order terms. Next, we   recall some of the results related to the stability  of two equations coupled through displacements. In  \cite{Alabau04}, Alabau considers coupled equations with only one boundary control where it is shown for different examples such as the wave equations or the Kirchhoff plates that the full system can be strongly stabilized  provided that the coupling parameter is sufficiently small. In  such a case, the author proves that the energy decays polynomially with explicit polynomial decay rate for sufficiently smooth solutions and these results are extended  to the case of two coupled wave equations with different speeds of propagation under a condition on the ratio of the two speeds and for $n$-dimensional intervals. In \cite{Alabau02}, Alabau et al.,  study the indirect internal stabilization of weakly coupled equations where the damping  is effective in the whole domain.  The authors prove that the  behaviour of the first  equation is sufficient to stabilize the total system and to get a polynomial decay for sufficiently smooth solutions. In  \cite{Alabau06},  Alabau and  L\'eautaud  extend the result of \cite{Alabau02} to a system of two coupled equations with a coupling operator. Under certain assumptions on the coupling operator, they  prove the polynomial stability  the system. Furthermore, we mention \cite{Alabau05}, \cite{ LiuRao01}, \cite{Fu01} and \cite{Soufyane02}  for  indirect stabilization of coupled equations via  zero order terms. 
 \\[0.1in]
Last but not least,  we  recall some  results  concerning the exponential or polynomial indirect stability of systems which arise   from physical problems. For example, 
 we quote \cite{Wehbe08,Alabau01,Fatori02, Guesmia01,RaoLiu03, Wehbe03,Wehbe07,   Wehbe01} for the Bresse system and  \cite{Wehbe05,Alabau03,Soufyane01,Wehbe06,Kim01,Messaoud01,
Santos03} for the Timoshenko system.  The Bresse system is usually considered in the study of elastic structures of the arcs type (see \cite{LagneseLeugering01}) while the  Timoshenko system is usually considered in describing the transverse vibration of a beam and it ignores damping effects of any nature (see \cite{Timoshenko01}).
\\[0.1in]
\noindent The aim of the present paper consists in studying the stability of the  indirectly damped System \eqref{Coupled Equation 1.1}. For this purpose, we write  System \eqref{Coupled Equation 1.1}  as the differential system
\begin{equation*}
U_t=\mathcal{A}U,
\end{equation*}
where 
\begin{equation*}
U=
\begin{pmatrix}
u
           \\ \noalign{\medskip}
 v
           \\ \noalign{\medskip}
 y
           \\ \noalign{\medskip}
 z
 \end{pmatrix}
 \ \ \ \text{and}\ \ \  \mathcal{A} U=
 \begin{pmatrix}
v
           \\ \noalign{\medskip}
 -aAu-A^\gamma v-\alpha z
           \\ \noalign{\medskip}
  z 
           \\ \noalign{\medskip}
  -Ay+\alpha v
\end{pmatrix},
\end{equation*}\\[0.1in]
where $U\in H$ and $\mathcal{A}$ is an unbounded operator on $H$. We use  $\left<\cdot , \cdot\right>$ and $\left\|\cdot\right\|$ to denote the inner produc  and the induced norm respectively on $H$. Since the resolvent of $A$ turns out to be compact in $H$, there exists a non decreasing  sequence $(\mu_n)_{n\geq1}$ tending to infinity and an orthonormal basis $(e_n)_{n\geq1}$ of $H$ such that, for $n\geq 1$,
$Ae_n=\mu^2_n e_n$. We assume the spectrum of $A$ is simple, i.e., the sequence $(\mu_n)_{n\geq1}$ is increasing.
Our goal in this paper is to establish the optimal stability of System \eqref{Coupled Equation 1.1} using the spectral method for the operator $ \mathcal{A}$. For this aim, we study the effect of both the fractional order damping of type $A^{\gamma}$ and the speeds of the two wave equations on this spectrum  and we prove that the latter is made of two branches $(\lambda_{1,n})_{n\in\mathbb{N}}$ and $(\lambda_{2,n})_{n\in\mathbb{N}}$, whose asymptotics, as $n$ tends to infinity, are given next.\\[0.1in]
\noindent \textbf{Case 1.} Assume  that $\gamma=0$. 
If $a=1$, i.e., when the two waves propagate with equal
speed,  we prove that the spectrum of $\mathcal{A}$ has an asymptotic expansion, as $n$ tends to infinity, given by 
\begin{equation*}
\lambda_{1,n}^\pm=\pm\ i\mu_n-\dfrac{1}{4}+\dfrac{1}{4}\sqrt{1-4\alpha^2}+o(1)
\  \ \ \text{ and } \ \  \
\lambda_{2,n}^\pm=\pm\ i\mu_n-\dfrac{1}{4}-\dfrac{1}{4}\sqrt{1-4\alpha^2}+o(1),
\end{equation*}
see Lemma \ref{Coupled Lemma 2.4}). 
Note here that if $4\alpha^2>1$, $\sqrt{1-4\alpha^2}$ actually denotes 
the imaginary number $i\sqrt{4\alpha^2-1}$. We then prove that the energy of  the system is (uniformly) exponentially stable. If $a\neq1$, i.e., when the waves propagate with different  speeds, we show that the spectrum of $\mathcal{A}$ has asymptotic expansion, as $n$ tends to infinity, given by
\begin{equation*}
\lambda_{1,n}^\pm=\pm\ i\sqrt{a}\mu_n-\frac{1}{2}+o(1)
\ \  \ \text{ and }\ \ \ 
\lambda_{2,n}^\pm=\pm i\Big(\mu_n-\frac{\alpha^2}{\left( a-1\right)\mu_n}\Big) -
\frac{\alpha^2}{2\left(a-1\right)^2\mu_n^{2}}+o(1),
\end{equation*}
see Lemma \ref{Coupled Lemma 2.12}).  Thus, the real part corresponding to the first branch of eigenvalues is uniformly bounded and the real part
 corresponding to the second branch of eigenvalues
is of magnitude $\mu_n^{-2}$. Therefore, we prove that the total energy decays at the optimal rate $1/t.$\\[0.1in]
\noindent \textbf{Case 2.} Assume  that  $\gamma<0$. If $a=1$, then the real parts of $\lambda_{1,n}^\pm$ and $\lambda_{2,n}^\pm$ are of magnitude $\mu_n^{2\gamma}$ (see Lemma \ref{Coupled Lemma 2.9}) and the total energy decays at the optimal rate $t^{\frac{1}{\gamma}}$. If $a\neq1$, then the real part corresponding to the first branch of eigenvalues is of magnitude $\mu_n^{2\gamma}$ and the real part corresponding to the second branch of eigenvalues is of magnitude $\mu_n^{2\gamma-2}$ (see Lemma \ref{Coupled Lemma 2.12}), yielding a decay rate of the optimal total energy equal to $t^{-\frac{1}{1-\gamma}}$. \\[0.1in]
\noindent From the above results, we deduce that  the maximum decay rate is achieved when $\gamma$ tends to zero. Therefore, a stronger damping term $A^{\gamma}u_t$ does not necessarily give a better  decay rate of the total energy, as it is expected. A good damping term should transmit the damping from one wave to another before the directly damped wave dies out or loses its energy. This effect of a good damping term  is interpreted by the real parts of the eigenvalues. Consequently, the results of this paper show that a suitable weaker damping term can compensate the lack of feedback on the second equation of System \eqref{Coupled Equation 1.1}. It seems interesting to consider coupled systems of the type \eqref{Coupled Equation 2.1}  with different operators $A_1$, $A_2$. Indeed the same results could be obtained without essential difficulty in the case $A_2 = A^2_1$. But in general we can no longer calculate explicitly the eigenvalues as in Lemmas \ref{Coupled Lemma 2.4}, \ref{Coupled Lemma 2.9}, \ref{Coupled Lemma 2.12}.
 \\[0.1in]

\noindent This paper is organized as follows. In Section \ref{Coupled Section 2}, we set the framework of System \eqref{Coupled Equation 1.1}  and we establish the characteristic equation satisfied by the eigenvalues of the operator $\AA$. Next, in  Section \ref{Coupled Section 3}, relying on the spectrum method,  we prove the exponential stability of System \eqref{Coupled Equation 2.1} when $a=1$ and $\gamma=0$. In Section \ref{Coupled Section 4}, we consider the other cases of $a$ and $\gamma$.  We prove the optimal polynomial energy decay rate of type $t^{-{\delta\left(\gamma\right)}}$ of System \eqref{Coupled Equation 2.1},  where
\begin{equation*}
\delta\left(\gamma\right)=
\left\{
\begin{array}{lll}

-\frac{1}{\gamma},&   \text{if }   a=1\text{ and } \gamma<0, 
          
           \\  \noalign{\medskip}
           
 \frac{1}{1-\gamma}, & \text{if }   a\neq1\text{ and } \gamma\leq0.
 
\end{array}
\right.
\end{equation*}
Finally, in Section \ref{Coupled Section 5}, we examine some  applications for our study.
\section{Characteristic equation and Riesz basis method}\label{Coupled Section 2}
\noindent In this paper, we consider the following abstract system of second order evolution equations
given by 
\begin{equation}\label{Coupled Equation 2.1}
\left\{
\begin{array}{lll}

\utt+aAu+A^{\gamma}\ut+\alpha \yt=0,

          \\ \noalign{\medskip}
\ytt+Ay-\alpha \ut=0,

\end{array}
\right.
\end{equation}
where $a>0,$ $  \gamma\leq0,\ \alpha\in\mathbb{R}^*$, the operator $A$ is a self-adjoint coercive operator with compact resolvent in a separable Hilbert space $H$.
{Let us define the energy space}
\begin{equation*}
\HH=
\Dah
\times 
H
\times
 \Dah
 \times
  H
\end{equation*}
equipped with the following norm
\begin{equation*}
\left\|\left(u,v,y,z\right)\right\|^2_\HH
=
a\left\|A^{\frac{1}{2}}u\right\|^2
+
\left\|A^{\frac{1}{2}}y\right\|^2
+
\left\|v\right\|^2+\left\|z\right\|^2,
\end{equation*}
where  $\|\cdot\|$  denotes   the  norm in $H$. We define the linear unbounded operator $\AA$ in $\HH$ by
\begin{equation*}
\DA
=
\left\{
U=\left(u,v,y,z\right)^\mathsf{T}\in\HH; \ 
  v,z\in \Dah
,\ 
u,y\in \Da
\right\},
\end{equation*}
and 
\begin{equation*}
\AA
\begin{pmatrix}
u
          \\ \noalign{\medskip}
v
          \\ \noalign{\medskip}
y
          \\ \noalign{\medskip}
z
\end{pmatrix}
=
\begin{pmatrix}
v
          \\ \noalign{\medskip}
-aAu-A^\gamma v-\alpha z
          \\ \noalign{\medskip}
z 
          \\ \noalign{\medskip}
-Ay+\alpha v
\end{pmatrix}.
\end{equation*}
Therefore, we can write System \eqref{Coupled Equation 2.1} as an evolution equation
\begin{equation}\label{eq:mathcalA}
\left\{
\begin{array}{lll}

U_t(x,t)=\AA U(x,t),

       \\ \noalign{\medskip}
       
U\left(x,0\right)=u_0(x),

\end{array}
\right. 
\end{equation}
 where $u_0=\left(u_0,v_0,y_0,z_0\right)^{\mathsf{T}}\in \HH.$\\[0.1in]
\noindent One clearly has that $\mathcal{A}$ is a maximal dissipative operator on $\mathcal{H}$ and, thanks to  the Lumer-Phillips theorem (see \cite{LiuZheng01, Pazy01}), we deduce that $\mathcal{A}$ generates a  $C_0$-semigroup of contractions $e^{t\mathcal{A}}$ in $\mathcal{H}$ and therefore \eqref{Coupled Equation 2.1} is well-posed.
Moreover, the energy of System \eqref{Coupled Equation 2.1} is given by
\begin{equation*}
\Et
=
\frac{1}{2}\left(
a\left\|A^{\frac{1}{2}}u\right\|^2
+
\left\|A^{\frac{1}{2}}y\right\|^2
+
\left\|\ut\right\|^2
+
\left\|\yt\right\|^2
\right),
\end{equation*}
where
\begin{equation*}
E^\prime\left(t\right)
=
-\left\|A^{\frac{\gamma}{2}}\ut\right\|^2
\leq0.
\end{equation*}
Hence, the energy of System \eqref{Coupled Equation 2.1} is decaying. Before starting the main results of this work, we introduce here the notions of stability that we encounter in this work.
\begin{defi}\label{Chapter pr-35}
{Assume that $\mathcal{A}$ is the generator of a C$_0$-semigroup of contractions $e^{t\mathcal{A}}$  on a Hilbert space  $\mathcal{H}$. The  $C_0$-semigroup $e^{t\mathcal{A}}$  is said to be
\begin{enumerate}
\item[1.]  Exponentially (or uniformly) stable if there exist two positive constants $M$ and $\epsilon$ such that
\begin{equation}\label{eq:ES}
\|e^{t\mathcal{A}}x_0\|_{\mathcal{H}} \leq Me^{-\epsilon t}\|x_0\|_{\mathcal{H}}, \quad
\forall\  t>0,  \ \forall \ x_0\in {\mathcal{H}}.
\end{equation}
\item[2.] Polynomially stable if there exists two positive constants $C$ and $\alpha$ such that
\begin{equation}\label{eq:PS}
 \|e^{t\mathcal{A}}x_0\|_{\mathcal{H}}\leq C t^{-\alpha}\|\mathcal{A}x_0\|_{\mathcal{H}},  \quad\forall\ 
t>0,  \ \forall \ x_0\in D\left(\mathcal{A}\right).
\end{equation}
In that case, one says that solutions of \eqref{eq:mathcalA} decay at a rate $t^{-\alpha}$.
\noindent The  $C_0$-semigroup $e^{t\mathcal{A}}$  is said to be  polynomially stable with optimal decay rate $t^{-\alpha}$ (with $\alpha>0$) if it is polynomially stable with decay rate $t^{-\alpha}$ and, for any $\varepsilon>0$ small enough, solutions of \eqref{eq:mathcalA} do not decay at a rate $t^{-(\alpha-\varepsilon)}$.
\end{enumerate}}
\xqed{$\square$}
\end{defi}
\noindent Note that, in the definition of polynomially stable, one can replace $\mathcal{A}$ by $\mathcal{A}^{\theta}$ for some positive real number $\theta$ and the constants $C,\alpha$ hence depend on $\theta$.
\begin{defi}
Let $(e^{t\mathcal{A}})_{t\geq0}$  be a $C_{0}$-semigroup of contractions generated by the operator $\mathcal{A}$ on a Hilbert space $\mathcal{H}$. Let  $(\lambda_{k,n})_{1\leq k\leq K,~n\geq 1}$ denotes the kth branch of eigenvalues of $\mathcal{A}$ and $\{ e_{k,n} \}_{1\leq k\leq K,~n\geq 1}$ the system of eigenvectors which forms a Riesz basis in $\mathcal{H}$. Then the fractional power $\mathcal{A}^{\theta}$ of $\mathcal{A}$ with $\theta\in \mathbb{R}$ is defined by
\begin{equation*}
D\left(\mathcal{A}^{\theta}\right)=\left\{  u\in \mathcal{H}:\   \sum_{k=1}^K\sum_{n\geq 1}\left| \lambda_{k,n}^{\theta}\right|^2\left|\left<u,e_{k,n}\right>_\mathcal{H}\right|^2 <\infty    \right\}
\end{equation*}
and for all $u\in D\left(\mathcal{A}^{\theta}\right)$, we have
\begin{equation*}
\mathcal{A}^{\theta}u=\sum_{k=1}^K\sum_{n\geq 1} \lambda_n^{\theta}\left<u,e_{k,n}\right>_\mathcal{H} e_{k,n}.\end{equation*}\xqed{$\square$}
\end{defi}
\noindent Our subsequent findings on exponential stability will rely on the following result from \cite{CurtainZwart01,LuoBao01}, which gives necessary and sufficient conditions for a semigroup to be exponentially stable.
\begin{prop}\label{Chapter pr-43}$\left(\textbf{cf.  \cite{CurtainZwart01,LuoBao01}}\right)$
{Let $\left(\mathcal{A},D\left(\mathcal{A}\right)\right)$  be an unbounded linear operator on $\mathcal{H}$ with compact resolvent. Assume that   $\mathcal{A}$ is the infinitesimal generator of a  $C_0-$semigroup of contractions $(e^{t\mathcal{A}})_{t\geq0}$. Moreover, suppose that the eigenvectors and the root vectors of $\mathcal{A}$ form a Riesz basis in $\mathcal{H}$ and that the multiplicity of the eigenvalues of $\mathcal{A}$ are uniformly bounded.  Then, $(e^{t\mathcal{A}})_{t\geq0}$ is  exponentially stable if and only if its spectral bound $s\left(\mathcal{A}\right)$,} defined as
$$s\left(\mathcal{A}\right)=\sup\left\{\Re\left(\lambda\right):\ \lambda\in \sigma\left(\mathcal{A}\right) \right\},$$ is negative.
\xqed{$\square$}
\end{prop}
\noindent If the semigroup fails to be exponentially stable, we search for another type of decay rate such polynomial stability. In that case, the following proposition from \cite{RaoLoreti01} provides a useful way to even characterize optimal polynomial stability.
\begin{prop}\label{Chapter pr-45}
	{ $\left(\textbf{cf. Theorem 2.1   in \cite{RaoLoreti01}}\right)$.  Let $(e^{t\mathcal{A}})_{t\geq0}$  be a $C_{0}$-semigroup of contractions generated by the operator $\mathcal{A}$ on a Hilbert space $\mathcal{H}$. Let $(\lambda_{k,n})_{1\leq k\leq K,~n\geq 1}$ denotes the $k$-th branch of eigenvalues of $\mathcal{A}$ and $\{ e_{k,n} \}_{1\leq k\leq K,~n\geq 1}$ the system of eigenvectors which forms a Riesz basis in $\mathcal{H}$. Assume that for each $1\leq k\leq K$ there exist a positive sequence $(\mu_{k,n})_{n\geq 1}$ tending to infinity and two positive constants $\alpha_{k}\geq 0,$ $\beta_{k}>0$ such that
	\begin{equation*}
	\Re (\lambda_{k,n})\leq -\dfrac{\beta_{k}}{\mu_{k,n}^{\alpha_{k}}}\quad\text{and}\quad |\text{Im} (\lambda_{k,n})|\geq \mu_{k,n}\quad \forall n\geq 1.
	\end{equation*}
	Then, for every $\theta >0,$ there exists a constant $M>0$ such that, for every $u_{0}\in D(\mathcal{A}^{\theta})$, one has 
	\begin{equation*}
	\|e^{t\mathcal{A}}u_{0}\|_{\mathcal{H}}\leq \|\mathcal{A}^{\theta}u_{0}\|_{\mathcal{H}}\dfrac{M}{t^{\theta \delta}} \quad \forall t>0,
	\end{equation*}
	where $\delta$ is given by
	\begin{equation}\label{delta}
	\delta := \min_{1\leq k\leq K} \dfrac{1}{\alpha_{k}}=\dfrac{1}{\alpha_{l}}.
	\end{equation}
	Moreover, if there exists two constants $c_{1}>0$, $c_{2}>0$ such that
	\begin{equation*}
	\Re (\lambda_{k,n})\geq -\dfrac{c_{1}}{\mu_{k,n}^{\alpha_{l}}}\quad\text{and}\quad |\text{Im} (\lambda_{k,n})|\leq c_{2}\mu_{k,n}\quad \quad 1\leq k\leq K,\quad \forall n\geq 1,
	\end{equation*}
	then $(e^{t\mathcal{A}})_{t\geq 0}$  is polynomially stable with optimal decay rate $t^{-\delta}$, where $\delta$ is given in \eqref{delta}.}\xqed{$\square$}
\end{prop}
\noindent In this work, to check the decay rate, we rely on the Riesz basis 
method in which we first determine the characteristic equation satisfied by the spectrum. Since the resolvent of $A$ is compact in $H$, there exists an increasing sequence $(\mu_n)_{n\geq1}$ tending to infinity and an orthonormal basis $(e_n)_{n\geq1}$ of $H$ such that
\begin{equation}\label{Coupled Equation 2.2}
Ae_n
=
\mu^2_n e_n\quad\forall\ n\geq 1.
\end{equation}
In turn, to study the spectrum of System \eqref{Coupled Equation 2.1}, let $\lambda$ be an eigenvalue of the operator $\AA$ and $U=(u, v, y, z)^{\mathsf{T}}$ a corresponding eigenvector. Therefore, we have
\begin{equation*}
\AA U
=
\lambda U.
\end{equation*}
Equivalently, we have the following system
\begin{equation}\label{Coupled Equation 2.3}
\left\{
\begin{array}{ll}

v=\lambda u,
        \\ \noalign{\medskip} 
-aAu-A^\gamma v-\alpha z=\lambda v,
        \\ \noalign{\medskip}
 z=\lambda y,
         \\  \noalign{\medskip}
  -Ay+\alpha v=\lambda z.
  
\end{array}
\right.
\end{equation}
Similar to the analysis done in \cite{RaoLoreti01}, we will see in Proposition \ref{Coupled Theorem 2.7} and in Proposition  \ref{Coupled Theorem 2.14new} that, for every $n\geq 1$, there exists $\left(B_n,C_n\right)\neq\left(0,0\right)$ such that the eigenvector $U$ of $\mathcal{A}$  is of the form
\begin{equation}\label{Coupled Equation 2.4}
u=B_n e_n,\quad v=\lambda B_n e_n,\quad y=C_n e_n,\quad z=\lambda C_n e_n.
\end{equation}
 Inserting \eqref{Coupled Equation 2.4} in \eqref{Coupled Equation 2.3} and using \eqref{Coupled Equation 2.2}, we obtain
\begin{equation}\label{Coupled Equation 2.5}
\left\{
\begin{array}{ll}

\displaystyle{ \left(a\mu^2_n+\lambda^2+\lambda \mu^{2\gamma}_n\right)B_ne_n +\alpha \lambda C_n e_n=0,}
    \\  \noalign{\medskip}
\displaystyle{-\alpha\lambda B_n e_n  +\left(\mu^2_n+ \lambda^2\right) C_n e_n=0,}

\end{array}
\right.
\end{equation}
which  has a non-trivial solution $\left(B_n,C_n\right)\neq\left(0,0\right)$ if and only if $\lambda$ is a solution of the equation 
\begin{equation}\label{charact eq}
 a\mu^4_n+\lambda\left(\left(a+1\right)\lambda+\mu^{2\gamma}_n\right)\mu^2_n+\lambda^2\left(\lambda^2+\lambda\mu^{2\gamma}_n+\alpha^2\right)=0,
\end{equation}
that we refer to as the characteristic equation associated with the eigenvalue $\mu_n^2$ of $A$. The four roots of this equation are eigenvalues of $\mathcal{A}$
and called the eigenvalues of $\mathcal{A}$ corresponding to $\mu_n$ of $A$. We also have the following result.
\begin{lem}\label{Coupled Lemma 2.1}
{ Let $\lambda_n=\lambda_{j,n}^\pm,\ j=1,2$ be one of the fourth eigenvalues of $\mathcal{A}$ corresponding to $\mu_n$. Then, there exists two positive constants $m,\ M$, such that, for $n$ large enough,   
 \begin{equation}\label{Coupled Equation 2.9}
m\leq\left\vert\frac{\lambda_n}{\mu_n}\right\vert\leq M.
\end{equation}}
 \end{lem}
\begin{proof} Set $Z_n=\frac{\lambda_n}{\mu_n}$. Then, from \eqref{charact eq}, one has that $Z_n$ is one of the four roots of the polynomial $f_n$ of degree four given by 
$$
f_n(Z)=Z^4+\mu_n^{2\gamma-1}Z^3+(a+1+\frac{\alpha^2}{\mu_n^2})Z^2+\mu_n^{2\gamma-1}Z+a.
$$
Let $g$ be the the polynomial of degree four given by 
$g(Z)=Z^4+(a+1)Z^2+a$, which has exactly four non zero roots. Since the coefficients of $f_n$ converge to those of $g$ as $n$ tends to infinity, one gets the result.

 \end{proof}	
\section{Exponential stability}\label{Coupled Section 3}
\noindent In this Section,  we consider the case where $a=1$ and $\gamma=0$.
Our main result is the following theorem.
\begin{thm}\label{Coupled Theorem}
If  $a=1$ and $\gamma=0$, then System \eqref{Coupled Equation 2.1} is exponentially stable.
\end{thm}
\noindent For the proof of Theorem \ref{Coupled Theorem}, we first need  to study the asymptotic behaviour of the spectrum of $\mathcal{A}$ and, in that direction, we have the following Lemma.
\begin{lem}\label{Coupled Lemma 2.4}
{Assume that $a=1$ and $\gamma=0$. Then, for  $n\geq 1$ large enough, the four eigenvalues of $\mathcal{A}$ corresponding to the eigenvalue $\mu_n^2$ of $A$ and denoted $\lambda_{1,n}^{\pm},\ \lambda_{2,n}^{\pm}$,  satisfy the following asymptotic expansions\\[0.1in]
\noindent\textbf{Case 1.} If $0<\alpha^2\leq \frac{1}{4}$, then 
\begin{equation}\label{Coupled Equation 2.17}
\left\{\begin{array}{ll}

\displaystyle{\lambda_{1,n}^\pm=\pm\ i\mu_n-\dfrac{1}{4}+\dfrac{1}{4}\sqrt{1-4\alpha^2}+O\left(\frac{1}{\mu_n}\right),}

         \\ \noalign{\medskip} 

\displaystyle{\lambda_{2,n}^{\pm}=\pm\ i\mu_n-\dfrac{1}{4}-\dfrac{1}{4}\sqrt{1-4\alpha^2}+O\left(\frac{1}{\mu_n}\right)}.

\end{array}
\right.
\end{equation}
\noindent\textbf{Case 2.} If $\alpha^2>\frac{1}{4}$, then 
\begin{equation}\label{Coupled Equation 2.17new}
\left\{\begin{array}{ll}

\displaystyle{\lambda_{1,n}^\pm=\pm\ i\mu_n-\dfrac{1}{4}+\dfrac{i}{4}\sqrt{4\alpha^2-1}+O\left(\frac{1}{\mu_n}\right),}

         \\ \noalign{\medskip} 

\displaystyle{\lambda_{2,n}^{\pm}=\pm\ i\mu_n-\dfrac{1}{4}-\dfrac{i}{4}\sqrt{4\alpha^2-1}+O\left(\frac{1}{\mu_n}\right)}.

\end{array}
\right.
\end{equation}}
\end{lem}
\begin{proof}
We divide the proof into two cases.\\[0.1in]
\noindent\textbf{Case 1.} If $0<\alpha^2\leq \frac{1}{4}$, from \eqref{charact eq}, we get that 
\begin{equation}\label{Coupled Equation 2.15}
\lambda^2_{1,n}+\frac{1-\sqrt{1-4\alpha^2}}{2}\lambda_{1,n}+\mu_n^2=0
\ \ \ \text{and}\ \ \ 
\lambda^2_{2,n}+\frac{1+\sqrt{1-4\alpha^2}}{2}\lambda_{2,n}+\mu_n^2=0.
\end{equation}
For $n$ large, solving equation \eqref{Coupled Equation 2.15}, we obtain
\begin{equation}\label{Coupled Equation 2.22}
\left\{\begin{array}{ll}

\displaystyle{\lambda_{1,n}^\pm=\dfrac{1}{4}\left(-1+\sqrt{1-4\alpha^2}\right)\pm\frac{i}{4}\sqrt{16\mu_n^2-\left(-1+\sqrt{1-4\alpha^2}\right)^2}},

                \\ \noalign{\medskip} 

\displaystyle{\lambda_{2,n}^\pm=\dfrac{1}{4}\left(-1-\sqrt{1-4\alpha^2}\right)\pm\frac{i}{4}\sqrt{16\mu_n^2-\left(-1-\sqrt{1-4\alpha^2}\right)^2}}.

\end{array}
\right. 
\end{equation}
Moreover, we have
\begin{equation}\label{Coupled Equation 2.23}
\frac{i}{4}\left[16\mu^2_n-\left(-1\pm\sqrt{1-4\alpha^2}\right)^2\right]^{\frac{1}{2}}
=
   \pm i   \mu_n\left[1+O\left(\frac{1}{\mu_n^2}\right)\right]^{\frac{1}{2}}
=
  \pm i   \mu_n+O\left(\frac{1}{\mu_n}\right).
\end{equation}
Substituting \eqref{Coupled Equation 2.23} in \eqref{Coupled Equation 2.22}, we obtain \eqref{Coupled Equation 2.17}. \\[0.1in]
\textbf{Case 2.} If $\alpha^2> \frac{1}{4}$, from \eqref{charact eq}, we obtain
\begin{equation}\label{Coupled Equation 2.15-new2}
\lambda^2_{1,n}+\frac{1-i\sqrt{4\alpha^2-1}}{2}\lambda_{1,n}+\mu_n^2=0
\ \ \ \text{and}\ \ \ 
\lambda^2_{2,n}+\frac{1+i\sqrt{4\alpha^2-1}}{2}\lambda_{2,n}+\mu_n^2=0.
\end{equation}
then for $n$ large, solving equation \eqref{Coupled Equation 2.15-new2}, we obtain
\begin{equation}\label{Coupled Equation 2.22newone}
\begin{array}{ll}

\displaystyle{\lambda_{1,n}^\pm=\frac{1}{4}\left(-1\mp\sqrt{2
{{\sqrt{16\mu^4_n+4(2\alpha^2-1)\mu^2_n +\alpha^4}-8\mu^2_n-2\alpha^2+1}}}\right)  }\\ \\ \hspace{4cm} \displaystyle{+\frac{i}{4}\left(\sqrt{4\alpha^2-1}\pm\sqrt{2
{{\sqrt{16\mu^4_n+4(2\alpha^2-1)\mu^2_n +\alpha^4}+8\mu^2_n+2\alpha^2-1}}}\right)}

\end{array} 
\end{equation}
and
\begin{equation}\label{Coupled Equation 2.22newtwo}
\begin{array}{ll}

\displaystyle{\lambda_{2,n}^\pm=\frac{1}{4}\left(-1\pm\sqrt{2
{{\sqrt{16\mu^4_n+4(2\alpha^2-1)\mu^2_n +\alpha^4}-8\mu^2_n-2\alpha^2+1}}}\right)  }\\ \\ \hspace{4cm} \displaystyle{+\frac{i}{4}\left(-\sqrt{4\alpha^2-1}\pm\sqrt{2
{{\sqrt{16\mu^4_n+4(2\alpha^2-1)\mu^2_n +\alpha^4}+8\mu^2_n+2\alpha^2-1}}}\right)},

\end{array} 
\end{equation}
since
\begin{equation}\label{Coupled Equation 2.23new}
\left\{\begin{array}{ll}
\displaystyle{
\sqrt{2\sqrt{16\mu^4_n+4(2\alpha^2-1)\mu^2_n +\alpha^4}-8\mu^2_n-2\alpha^2+1}=O\left(\frac{1}{\mu_n}\right),}\\ \\

\displaystyle{\sqrt{2\sqrt{16\mu^4_n+4(2\alpha^2-1)\mu^2 +\alpha^4}+8\mu^2_n+2\alpha^2-1}=4\mu_n+O\left(\frac{1}{\mu_n}\right),}
\end{array} \right.
\end{equation}
then inserting \eqref{Coupled Equation 2.23new} in \eqref{Coupled Equation 2.22newone} and \eqref{Coupled Equation 2.22newtwo}, we get \eqref{Coupled Equation 2.17new}. Thus, the proof is  complete.
\end{proof}
\noindent We next provide the form of the eigenvectors and root vectors of $\mathcal{A}$. We start with the  following Lemma.
 \begin{lem}\label{eigenvectors}
{ If $a=1$ and $\gamma=0$, then the eigenvectors of $\mathcal{A}$ take the following form.\\[0.1in]
\noindent\textbf{Case 1.} If $0<\alpha^2\leq \frac{1}{4}$, then we have
\begin{equation}\label{Coupled Equation 2.14}
\begin{array}{ll}
e_{1,n}^{\pm}=B_{1,n}^{\pm}\left(\dfrac{e_n}{\lambda_{1,n}^{\pm}}, e_n ,  \dfrac{\beta_1 e_n}{\lambda^\pm_{1,n}}, \displaystyle{\beta_1 e_n} \right)^{\top},\quad
   
e_{2,n}^{\pm}=C_{2,n}^{\pm}\left(\dfrac{\delta_1}{\lambda_{2,n}^{\pm}}e_n, \delta_1 e_n, \dfrac{e_n}{\lambda_{2,n}^{\pm}}, e_n    \right)^{\top},
\end{array}
\end{equation}
where $B_{1,n}^{\pm},C_{2,n}^{\pm}\in\mathbb{C}$ and $\beta_1=\dfrac{2\alpha }{-1+\sqrt{1-4\alpha^2} },\ \delta_1=-\dfrac{1+\sqrt{1-4\alpha^2}}{2\alpha}$.\\[0.1in]
\textbf{Case 2.} If $\alpha^2> \frac{1}{4}$, then we have
\begin{equation}\label{Coupled Equation 2.14new2}
\begin{array}{ll}
e_{1,n}^{\pm}=B_{1,n}^{\pm}\left(\dfrac{e_n}{\lambda_{1,n}^{\pm}}, e_n ,  \dfrac{\beta_3 e_n}{\lambda^\pm_{1,n}}, \displaystyle{\beta_3 e_n} \right)^{\top},\quad
   
e_{2,n}^{\pm}=C_{2,n}^{\pm}\left(\dfrac{\delta_3}{\lambda_{2,n}^{\pm}}e_n, \delta_3 e_n, \dfrac{e_n}{\lambda_{2,n}^{\pm}}, e_n    \right)^{\top},
\end{array}
\end{equation}
where $B_{1,n}^{\pm},C_{2,n}^{\pm}\in\mathbb{C}$ and $\beta_3=\dfrac{2\alpha }{-1+i\sqrt{4\alpha^2-1} },\ \delta_3=-\dfrac{1+i\sqrt{4\alpha^2-1}}{2\alpha}$.
}
 \end{lem}
\begin{proof}  
Let $\lambda_{1,n}^{\pm},\  \lambda_{2,n}^{\pm}$ be the solutions of \eqref{Coupled Equation 2.15}. Setting 
\begin{equation*}
B_{1,n}=\frac{B_{1,n}^\pm}{\lambda_{1,n}^\pm}\  
   \ \ \ \ \text{and} \ \ \   
 C_{2,n}=\frac{C_{2,n}^\pm}{\lambda_{2,n}^\pm}
\end{equation*}
 in \eqref{Coupled Equation 2.5}, we get
 \begin{equation*}
   C_{1,n}=\frac{\alpha}{\left(\lambda_{1,n}^{\pm}\right)^2+\mu^2_n } B_{1,n}^{\pm}
   \ \ \ \ \text{and} \ \ \   
 B_{2,n} =\frac{\left(\lambda_{2,n}^{\pm}\right)^2+\mu^2_n}{\alpha\left(\lambda_{2,n}^{\pm}\right)^2} C_{2,n}^{\pm}.
\end{equation*}
Therefore, from  \eqref{Coupled Equation 2.4}, we obtain
\begin{equation}\label{Coupled Equation 2.8}
\begin{array}{ll}
e_{1,n}^{\pm}=B_{1,n}^{\pm}\left(\dfrac{e_n}{\lambda_{1,n}^{\pm}}, e_n ,  \dfrac{\alpha e_n}{\left(\lambda_{1,n}^{\pm}\right)^2+\mu^2_n },  \dfrac{\alpha\lambda_{1,n}^{\pm}e_n}{\left(\lambda_{1,n}^{\pm}\right)^2+\mu^2_n }  \right)^{\top},

   \\ \\
   
e_{2,n}^{\pm}=C_{2,n}^{\pm}\left(\dfrac{\left(\lambda_{2,n}^{\pm}\right)^2+\mu^2_n}{\alpha\left(\lambda_{2,n}^{\pm}\right)^2}e_n, \dfrac{\left(\lambda_{2,n}^{\pm}\right)^2+\mu^2_n}{\alpha\lambda_{2,n}^{\pm}}e_n, \dfrac{e_n}{\lambda_{2,n}^{\pm}}, e_n    \right)^{\top},
\end{array}
\end{equation}
are the eigenvectors  corresponding  to the four eigenvalues $\lambda_{1,n}^{\pm},\  \lambda_{2,n}^{\pm}$, where  $B_{1,n}^{\pm},\  C_{2,n}^{\pm}\in\mathbb{C}$. We next divide the argument into two cases.\\[0.1in]
 \noindent\textbf{Case 1.} If $0<\alpha^2\leq \frac{1}{4}$, then from \eqref{Coupled Equation 2.15}, we get
\begin{equation}\label{Coupled Equation 2.16}
\frac{1}{\left(\lambda^\pm_{1,n}\right)^2+\mu_n^2}=\frac{2}{\left(-1+\sqrt{1-4\alpha^2}\right)\lambda_{1,n}^\pm}
\ \ \ \text{and} \ \ \
 \left(\lambda^\pm_{2,n}\right)^2+\mu_n^2=-\frac{1+\sqrt{1-4\alpha^2}}{2}\lambda_{2,n}^\pm.
\end{equation}
Inserting \eqref{Coupled Equation 2.16} in \eqref{Coupled Equation 2.8}, we get \eqref{Coupled Equation 2.14}.\\[0.1in]
\textbf{Case 2.} If $\alpha^2> \frac{1}{4}$, then from \eqref{Coupled Equation 2.15-new2}, we get
\begin{equation}\label{Coupled Equation 2.16new2}
\frac{1}{\left(\lambda^\pm_{1,n}\right)^2+\mu_n^2}=\frac{2}{\left(-1+i\sqrt{4\alpha^2-1}\right)\lambda_{1,n}^\pm}
\ \ \ \text{and} \ \ \
 \left(\lambda^\pm_{2,n}\right)^2+\mu_n^2=-\frac{1+i\sqrt{4\alpha^2-1}}{2}\lambda_{2,n}^\pm.
\end{equation}
Inserting \eqref{Coupled Equation 2.16new2} in \eqref{Coupled Equation 2.8}, we get \eqref{Coupled Equation 2.14new2}. Thus, the proof is  complete.
 \end{proof}
We now search for the asymptotic behaviour of the eigenvectors of $\mathcal{A}$. From Lemma \ref{Coupled Lemma 2.4}, we remark that if 
$\alpha^2=\frac{1}{4}$, we have double eigenvalues. In this case, we look for the corresponding root vectors. 
\begin{lem}\label{Coupled Lemma 2.4a}
{If $a=1$ and $\gamma=0$, then the   eigenvectors $e_{1,n}^{\pm},\ e_{2,n}^{\pm}$ of  $\mathcal{A}$ satisfy the following asymptotic expansion.\\[0.1in]
\noindent\textbf{Case 1.} If $0<\alpha^2< \frac{1}{4}$, then we have
\begin{equation}\label{Coupled Equation 2.18a}
\begin{array}{ll}
e_{1,n}^{\pm}=\displaystyle{\frac{1}{\beta^+_1}\left(\frac{e_n}{\pm\ i\mu_n}, e_n ,  \frac{\beta_1 e_n}{\pm  i   \mu_n},  \beta_1 e_n   \right)^{\top}}+O\left(\frac{1}{\mu_n^2}\right),

   \\ \\
   
e_{2,n}^{\pm}=\displaystyle{\frac{1}{\delta^+_1}\left(\frac{\delta_1 e_n}{\pm  i   \mu_n},\delta_1 e_n, \frac{e_n}{\pm  i   \mu_n}, e_n    \right)^{\top}}+O\left(\frac{1}{\mu_n^2}\right),
\end{array}
\end{equation}
 where $\beta_1=\dfrac{2\alpha }{-1+\sqrt{1-4\alpha^2} },\ \delta_1=-\dfrac{1+\sqrt{1-4\alpha^2}}{2\alpha},\  \beta^+_1=\sqrt{2+2\left|\beta_1\right|^2}$ and $\delta^+_1=\sqrt{2+2\left|\delta_1\right|^2}$.\\[0.1in]
\noindent\textbf{Case 2.} If $\alpha^2= \frac{1}{4}$, we suppose that  $\alpha=\frac{1}{2}$  since the analysis follows similarly if $\alpha=-\frac{1}{2}$, then the   eigenvectors $e_{n}^{\pm}$ of  $\mathcal{A}$ satisfy the following asymptotic expansion
\begin{equation}\label{Coupled Equation 2.20a}
e_{n}^{\pm}=\displaystyle{\frac{1}{2}\left(\frac{e_n}{\pm  i   \mu_n}, e_n ,  \frac{e_n}{\mp  i   \mu_n},  -  e_n  \right)^{\top}}+O\left(\frac{1}{\mu_n^2}\right),
\end{equation}
and the root vectors of $\mathcal{A}$ satisfy the following asymptotic expansion
\begin{equation}\label{Coupled Equation 2.21a}
\tilde{e}_{n}^{\pm}=\displaystyle{\frac{1}{\sqrt{2}}\left(\frac{e_n}{\pm  i   \mu_n}, e_n ,  0,  0  \right)^{\top}}+\left(O\left(\frac{1}{\mu_{n}^{2}}\right),0,0,O\left(\frac{1}{\mu_{n}}\right)\right)^{\top}.
\end{equation}
\noindent\textbf{Case 3.} If $\alpha^2>\frac{1}{4}$, then we have
\begin{equation}\label{Coupled Equation 2.18anew}
\begin{array}{ll}
e_{1,n}^{\pm}=\displaystyle{\frac{1}{\beta^+_3}\left(\frac{e_n}{\pm\ i\mu_n}, e_n ,  \frac{\beta_3 e_n}{\pm  i   \mu_n},  \beta_3 e_n   \right)^{\top}}+O\left(\frac{1}{\mu_n^2}\right),

   \\ \\
   
e_{2,n}^{\pm}=\displaystyle{\frac{1}{\delta^+_3}\left(\frac{\delta_3 e_n}{\pm  i   \mu_n},\delta_3 e_n, \frac{e_n}{\pm  i   \mu_n}, e_n    \right)^{\top}}+O\left(\frac{1}{\mu_n^2}\right),
\end{array}
\end{equation}
 where $\beta_3=\dfrac{2\alpha }{-1+i\sqrt{4\alpha^2-1} },\ \delta_3=-\dfrac{1+i\sqrt{4\alpha^2-1}}{2\alpha},\  \beta^+_3=\sqrt{2+2\left|\beta_3\right|^2}$ and $\delta^+_3=\sqrt{2+2\left|\delta_3\right|^2}$.}
\end{lem}
\begin{proof}
When $a=1$ and $\gamma=0$, we must now subdivide the proof into three cases.\\[0.1in]
\noindent\textbf{Case 1.} If $0<\alpha^2< \frac{1}{4}$, then from  \eqref{Coupled Equation 2.17}, we obtain
\begin{equation}\label{Coupled Equation 2.24}
\frac{1}{\lambda_{1,n}^{\pm}}=\frac{1}{\pm   i   \mu_n}+O\left(\frac{1}{\mu_n^2}\right)
 \ \ \ \text{and}\  \ \
\frac{1}{\lambda_{2,n}^{\pm}}=\frac{1}{\pm   i   \mu_n}+O\left(\frac{1}{\mu_n^2}\right).
\end{equation}
Setting $B_{1,n}^\pm=\frac{1}{\beta^+_1}$ and $C_{2,n}^\pm=\frac{1}{\delta^+_1}$ in  \eqref{Coupled Equation 2.14}, then using \eqref{Coupled Equation 2.24}, we get \eqref{Coupled Equation 2.18a}.\\[0.1in]
\noindent\textbf{Case 2.} If $\alpha^2=\frac{1}{4}$,  we suppose that  $\alpha=\frac{1}{2}$  since the analysis follows similarly if $\alpha=-\frac{1}{2}$. In this case, since \eqref{charact eq} admits two double solutions $\lambda_{n}^\pm$, the associated eigenvectors are given by  
\begin{equation}\label{Coupled Equation 2.26}
e_{n}^{\pm}=B_{n}^{\pm}
\left(
\frac{e_n}{\lambda_{n}^{\pm}}
          ,
e_n 
          ,
-\frac{ e_n}{\lambda_{n}^{\pm} }  
       ,
-e_n   
\right)^{\mathsf{T}},
\end{equation}
where $B_n^\pm\in\mathbb{C}.$ Furthermore, from \eqref{Coupled Equation 2.22}, we have
\begin{equation}\label{Coupled Equation 2.28}
\frac{1}{\lambda_{n}^{\pm}}
=
\frac{1}{\pm   i   \mu_n\left[1+O\left(\frac{1}{\mu_n}\right)\right]}
=
\frac{1}{\pm   i   \mu_n}+O\left(\frac{1}{\mu_n^2}\right).
\end{equation}
Setting $B_n^\pm=\frac{1}{2}$ in \eqref{Coupled Equation 2.26} and using  \eqref{Coupled Equation 2.28}, we get \eqref{Coupled Equation 2.20a}.
We next look for corresponding root vectors 
\begin{equation*}
\tilde{e}^{\pm}_{n}
=\left(
\tilde{u}^{\pm}_n
          ,
\tilde{v}^{\pm}_n
         ,
\tilde{y}^{\pm}_n 
         ,
\tilde{z}^{\pm}_n
\right)^{\mathsf{T}},
\end{equation*}
such that 
 \begin{equation*}
\left(\lambda_{n}^{\pm}I-\mathcal{A}\right)\tilde{e}^{\pm}_{n}
=
{e}^{\pm}_{n}.
\end{equation*}
Equivalently, we have
\begin{equation}\label{Coupled Equation 2.29}
\left\{
\begin{array}{ll}

\lambda^{\pm}_n \tilde{u}^{\pm}_n-\tilde{v}^{\pm}_n=\frac{e_n}{\lambda_{1,n}^{\pm}},

           \\ \noalign{\medskip}

 \lambda^{\pm}_n \tilde{v}^{\pm}_n+A\tilde{u}^{\pm}_n+ \tilde{v}^{\pm}_n+\alpha \tilde{z}^{\pm}_n=e_n ,
 
            \\ \noalign{\medskip} 
 
 \lambda^{\pm}_n \tilde{y}^{\pm}_n-\tilde{z}^{\pm}_n=-\frac{e_n}{\lambda_{1,n}^{\pm}},
 
            \\ \noalign{\medskip}
 
\lambda^{\pm}_n \tilde{z}^{\pm}_n +A\tilde{y}^{\pm}_n-\alpha \tilde{v}^{\pm}_n=-e_n.

\end{array}
\right.
\end{equation}
Setting $\tilde{u}^{\pm}_n=c^{\pm} e_n$ and $\tilde{y}^{\pm}_n=d^{\pm} e_n$ in \eqref{Coupled Equation 2.29}, we get
\begin{equation}\label{Coupled Equation 2.30}
\tilde{v}^{\pm}_n=\left( c^{\pm}\lambda^{\pm}_n -\frac{1}{\lambda_{n}^{\pm}}\right)e_n
\ \ \ \text{and}\ \ \ 
 \tilde{z}^{\pm}_n=\left( d^{\pm} \lambda^{\pm}_n  +\frac{1 }{\lambda_{n}^{\pm}}\right)e_n,
\end{equation}
where the constants $c^{\pm}$ and $d^{\pm}$ satisfy
\begin{equation}\label{Coupled Equation 2.31}
\left\{
\begin{array}{ll}

  \left(\mu^2_n+\left(\lambda_n^{\pm}\right)^2+\lambda_n^{\pm} \right)c^{\pm} +\alpha\lambda_n^{\pm}d^{\pm} =2+\frac{1}{2\lambda_n^{\pm}},
  
           \\ \noalign{\medskip}

 -\alpha\lambda_n^{\pm} c^{\pm} +\left(\mu^2_n+ \left(\lambda_n^{\pm}\right)^2\right)d^{\pm} =-2-\frac{1}{2\lambda_n^{\pm}}.
 
\end{array}
\right.
\end{equation}
Since $\lambda_n^{\pm}$ satisfy \eqref{charact eq}, then the first equation of \eqref{Coupled Equation 2.31} can be reduced to the second one.
 Therefore, taking
$$ d^{\pm} =0
 \ \ \ \text{and}\ \ \ 
 c^{\pm}=\frac{4\lambda_n^{\pm}+1}{\left( \lambda_n^{\pm}\right)^2}$$
in \eqref{Coupled Equation 2.30}, we get
\begin{equation}\label{Coupled Equation 2.32}
\tilde{e}_{n}^{\pm}=B_{n}^{\pm}\left(\dfrac{4\lambda_n^{\pm}+1}{ \left( \lambda_n^{\pm}\right)^2}e_n, 4e_n, 0, \dfrac{e_n }{\lambda_{n}^{\pm}} \right)^{\top},
\end{equation}
where $B^\pm_n\in\mathbb{C}.$ From \eqref{Coupled Equation 2.22} and \eqref{Coupled Equation 2.28}, we obtain
\begin{equation}\label{Coupled Equation 2.33}
\frac{4\lambda_n^{\pm}+1}{ \left( \lambda_n^{\pm}\right)^2}=\frac{4}{\pm  i   \mu_n}+O\left(\frac{1}{\mu_n^2}\right)
\ \ \ \text{and}\ \ \ 
\frac{1}{  \lambda_n^{\pm}}=O\left(\frac{1}{\mu_n}\right).
\end{equation}
Finally, setting $B^\pm_n=\frac{1}{4\sqrt{2}}$  in \eqref{Coupled Equation 2.32}, then using   \eqref{Coupled Equation 2.33}, we get \eqref{Coupled Equation 2.21a}. \\[0.1in]
\noindent\textbf{Case 3.} If $\alpha^2> \frac{1}{4}$, then from  \eqref{Coupled Equation 2.17new}, we obtain
\begin{equation}\label{Coupled Equation 2.24new}
\frac{1}{\lambda_{1,n}^{\pm}}=\frac{1}{\pm   i   \mu_n}+O\left(\frac{1}{\mu_n^2}\right)
 \ \ \ \text{and}\  \ \
\frac{1}{\lambda_{2,n}^{\pm}}=\frac{1}{\pm   i   \mu_n}+O\left(\frac{1}{\mu_n^2}\right).
\end{equation}
Setting $B_{1,n}^\pm=\frac{1}{\beta^+_3}$ and $C_{2,n}^\pm=\frac{1}{\delta^+_3}$ in  \eqref{Coupled Equation 2.14}, then using \eqref{Coupled Equation 2.24new}, we get \eqref{Coupled Equation 2.18anew}. Thus, the proof is  complete.
\end{proof}
\noindent Let now $E^{\pm}_{1,n},\ E^{\pm}_{2,n}$ be linearly independent eigenvectors of the decoupled system (corresponding to $\alpha=0$). Then one has
\begin{equation}\label{Coupled Equation 2.34}
E^{\pm}_{1,n}=\displaystyle{\frac{1}{\sqrt{2}}
\left(
\frac{e_n}{\pm  i   \sqrt{a}\mu_n},
e_n,0,0\right)^{\top}}           
\ \ \ \text{ and }\ \ \ 
E^{\pm}_{2,n}=\displaystyle{\frac{1}{\sqrt{2}}
\left(0,0,\frac{e_n}{\pm  i   \mu_n}, e_n\right)^{\top}}.
\end{equation}
Moreover, for every $n\geq1$, define 
\begin{equation}\label{eq:Wn}
W_n
=
\text{span}\left\{E_{1,n}^+,E_{1,n}^-,E_{2,n}^+,E_{2,n}^-\right\},
\end{equation}
and
\begin{equation}\label{eq:Vn}
V_n
=
{span}\left\{e_{1,n}^+,e_{1,n}^-,e_{2,n}^+,e_{2,n}^-\right\},
\quad
\tilde{V}_n
=
{span}\left\{e_{n}^+,e_{n}^-,\tilde{e}_{n}^+,\tilde{e}_{n}^-\right\}.
\end{equation}
From what precedes, one gets the following corollary.
\begin{cor}\label{Coupled Lemma 2.5}
{If $a=1$ and $\gamma=0$, then from Lemma \ref{Coupled Lemma 2.4a}, the following relationship holds.\\[0.1in]
\noindent\textbf{Case 1.} If $0<\alpha^2<\frac{1}{4}$, then
\begin{equation}\label{Coupled Equation 2.35}
\left(e_{1,n}^+,e_{1,n}^-,e_{2,n}^+,e_{2,n}^-\right)
=
\left(E_{1,n}^+,E_{1,n}^-,E_{2,n}^+,E_{2,n}^-\right)L_n,
\end{equation}
where
\begin{equation}\label{Coupled Equation 2.36}
L_n=\sqrt{2}\begin{pmatrix}
\frac{1}{\beta^+_1} &0&\frac{\delta_1}{\delta^+_1}&0
           \\ \noalign{\medskip}
0 &\frac{1}{\beta^+_1}&0&\frac{\delta_1}{\delta^+_1}
           \\ \noalign{\medskip}
\frac{\beta_1}{\beta^+_1} &0&\frac{1}{\delta^+_1}&0
           \\ \noalign{\medskip}
0&\frac{\beta_1}{\beta^+_1}&0 &\frac{1}{\delta^+_1}
           \\ \noalign{\medskip}
\end{pmatrix}+O\left(\frac{1}{\mu_n^2}\right).
\end{equation}
\noindent\textbf{Case 2.} If $\alpha^2= \frac{1}{4}$, then
\begin{equation}\label{Coupled Equation 2.37}
\left(e_{n}^+,e_{n}^-,\tilde{e}_{n}^+,\tilde{e}_{n}^-\right)
=
\left(E_{1,n}^+,E_{1,n}^-,E_{2,n}^+,E_{2,n}^-\right)\tilde{L}_n,
\end{equation}
where
\begin{equation}\label{Coupled Equation 2.38}
\tilde{L}_n=\frac{1}{2}\begin{pmatrix}
\sqrt{2} &0&2&0
           \\ \noalign{\medskip}
0 &\sqrt{2}&0&2
           \\ \noalign{\medskip}
-\sqrt{2} &0&0&0
           \\ \noalign{\medskip}
0&-\sqrt{2} &0 &0
           \\ \noalign{\medskip}
\end{pmatrix}
+O\left(\frac{1}{\mu_n}\right).
\end{equation}}
\noindent\textbf{Case 3.} If $\alpha^2>\frac{1}{4}$, then
\begin{equation}\label{Coupled Equation 2.35new}
\left(e_{1,n}^+,e_{1,n}^-,e_{2,n}^+,e_{2,n}^-\right)
=
\left(E_{1,n}^+,E_{1,n}^-,E_{2,n}^+,E_{2,n}^-\right)\tilde{\tilde{L}}_n,
\end{equation}
where
\begin{equation}\label{Coupled Equation 2.36new}
\tilde{\tilde{L}}_n=\sqrt{2}\begin{pmatrix}
\frac{1}{\beta^+_3} &0&\frac{\delta_3}{\delta^+_3}&0
           \\ \noalign{\medskip}
0 &\frac{1}{\beta^+_3}&0&\frac{\delta_3}{\delta^+_3}
           \\ \noalign{\medskip}
\frac{\beta_3}{\beta^+_3} &0&\frac{1}{\delta^+_3}&0
           \\ \noalign{\medskip}
0&\frac{\beta_3}{\beta^+_3}&0 &\frac{1}{\delta^+_3}
           \\ \noalign{\medskip}
\end{pmatrix}+O\left(\frac{1}{\mu_n^2}\right).
\end{equation}
\xqed{$\square$}
\end{cor}
 \noindent In the sequel, our aim is to prove the following proposition.
\begin{prop}\label{Coupled Theorem 2.7}
Suppose  that $a=1$ and $\gamma=0.$ Then the following holds true.
\begin{description}
\item[$(i)$] If $\alpha^2\neq \frac{1}{4}$,
then the set $\left\{e_{1,n}^+,e_{1,n}^-,e_{2,n}^+,e_{2,n}^-\right\}_{n\geq1}$ of eigenvectors of $\mathcal{A}$ forms a Riesz basis in $\mathcal{H}$. In particular, all eigenvectors of $\mathcal{A}$ are of the form given in \eqref{Coupled Equation 2.4}.
\item[$(ii)$] If $\alpha^2=\frac{1}{4}$, then the set $\left\{e_{n}^+,e_{n}^-,\tilde{e}_{n}^+,\tilde{e}_{n}^-\right\}_{n\geq1}$ of eigenvectors and root vectors of $\mathcal{A}$ forms a Riesz basis in $\mathcal{H}$.  In particular, all eigenvectors of $\mathcal{A}$ are of the form given in \eqref{Coupled Equation 2.4}.
\end{description}
\end{prop}
\noindent To prove Proposition  \ref{Coupled Theorem 2.7}, we first recall Lemma  3.1 in \cite{RaoLoreti01}.
\begin{prop}\label{Chapter pr-57}
 {$\left(\textbf{Lemma 3.1 in \cite{RaoLoreti01}}\right)$ Let $\left\{X_n\right\}_{n\geq1}$ be a Riesz basis of subspaces in a Hilbert space $H$ and $\left\{Y_n\right\}_{n\geq1}$ a Riesz sequence of subspaces in $H.$ Assume that there exist a sequence of isomorphisms $\left\{L_n\right\}_{n\geq1}$ from $X_n$ onto $Y_n$ and positive constants $m,\ M$  independent of $n$ such that
\begin{equation}\label{Coupled Equation 2.39}
\forall\ x_n\in X_n,\ \forall\ n\geq1,\quad
m\left\|x_n\right\|
\leq
\left\|L_nx_n\right\|
\leq
 M\left\|x_n\right\|.
\end{equation}
Assume furthermore that there exist a Riesz basis $\left\{f_{n,i}\right\}_{1\leq i\leq I_n}\ \left(I_n \leq +\infty\right)$ in each $X_n$ and positive constants $c,\ C$  independent of $n$ such that, for every $x_n=\sum_{i=1}^{I_n}\alpha_{n,i}f_{n,i}$, one has
\begin{equation}\label{Coupled Equation 2.40}
c\sum_{i=1}^{I_n}\left|\alpha_{n,i}\right|^2
\leq
\left\|x_n\right\|^2
\leq 
C\sum_{i=1}^{I_n}\left|\alpha_{n,i}\right|^2.
\end{equation}
Then the sequence
\begin{equation}\label{Coupled Equation 2.41}
g_{n,i}=L_n f_{n,i}
\quad\forall\ n\geq1,\ 1\leq i\leq I_n.
\end{equation}
forms a Riesz basis in $H.$}
\xqed{$\square$}
\end{prop}
\noindent\textbf{Proof of Proposition  \ref{Coupled Theorem 2.7}.} First, we prove that $\left\{W_n\right\}_{n\geq1}$, defined in \eqref{eq:Wn}, is a  Riesz basis of subspaces of $\mathcal{H}$. Let $$U=\left(u,v,y,z\right)^{\top}\in\mathcal{H}.$$ Since $\left(e_n\right)_{n\geq1}$ is a Hilbert basis of $H$, then 
\begin{equation*}
u=\sum_{n\geq1}u_n e_n
\quad 
v=\sum_{n\geq1}v_n e_n
\quad 
y=\sum_{n\geq1}y_n e_n
\quad 
z=\sum_{n\geq1}z_n e_n.
\end{equation*}
Hence, 
\begin{equation*}
\begin{array}{llll}
U
&=
\displaystyle{\sum_{n\geq1}\left(u_ne_n,v_ne_ny_ne_n,z_ne_n\right)^{\mathsf{T}}}

\\ \\ 

&=
\displaystyle{\frac{1}{\sqrt{2}}\sum_{n\geq1}\big[\left( v_n+  i   \mu_nu_n\right)E_{1,n}^{+}+\left( v_n-  i   \mu_nu_n\right)E_{1,n}^{-}+\left( z_n+  i   \mu_ny_n\right)E_{2,n}^{+}+\left( z_n-  i   \mu_ny_n\right)E_{2,n}^{-}\big]}.

\end{array}
\end{equation*}
Therefore, for every $U\in\mathcal{H}$, $U$ can be  written as $\displaystyle{\sum_{n\geq1}}U_n$ with $U_n\in W_n$. Moreover, for every pair of different positive integers $(n,m)$, we have that $U_n$ and $U_{m}$ are perpendicular since $W_n$ and $W_{m}$ are.
Therefore $\left\|U\right\|^2_\HH=\displaystyle{\sum_{n\geq1}}\left\|U_n\right\|^2$ and $U$ can be uniquely written  as $\displaystyle{\sum_{n\geq1}}U_n$. This yields that  $\left\{W_n\right\}_{n\geq1}$ is a Riesz basis of subspaces in $\mathcal{H}.$
We prove similarly that $\left\{V_n\right\}_{n\geq1}$ and  $\left\{\tilde{V}_n\right\}_{n\geq1}$ form a Riesz sequence of subspaces in $\mathcal{H}$.
Next, we divide the proof into three cases: $0<\alpha^2<\frac{1}{4}$, 
$ \alpha^2=\frac{1}{4}$ and If $\alpha^2>\frac{1}{4}$. Since the argument is entirely similar for the three cases, we only provide one of them.\\[0.1in]
If $\alpha^2\neq \frac{1}{4}$, then from Corollary \ref{Coupled Lemma 2.5}, we remark that $L_n$ has a constant leading term which is invertible. This, together with the fact that $L_n$ is invertible for every $n\geq 1$, implies  condition \eqref{Coupled Equation 2.39}. Moreover,  the condition \eqref{Coupled Equation 2.40} is satisfied since $\{E_{1,n}^+,\ E_{1,n}^-,\ E_{2,n}^+,\ E_{2,n}^-\}$ forms a  Hilbert basis in the subspace $W_n.$ Then applying Proposition  \ref{Chapter pr-57}, we obtain that the system of eigenvectors $\left\{e_{1,n}^+,e_{1,n}^-,e_{2,n}^+,e_{2,n}^-\right\}_{n\geq1}$ forms a Riesz basis in $\mathcal{H}$. Thus, the proof is  complete.
\xqed{$\square$}\\[0.1in]
\noindent \textbf{Proof of Theorem \ref{Coupled Theorem}.} From Lemma \ref{Coupled Lemma 2.4}, the large eigenvalues $\left\{\lambda_{k,n}^{\pm}\right\}_{k=1,2,\ n\geq n_0}$ of $\mathcal{A}$  satisfy the following estimation
\begin{equation*}
\Re\left\{\lambda_{k,n}^{\pm}\right\}
=\left\{\begin{array}{ll}
\displaystyle{-\frac{1}{4}\pm \frac{1}{4}\sqrt{1-4\alpha^2}+O\left(\frac{1}{\mu_n}\right)},&\text{if } \displaystyle{0<\alpha^2<\frac{1}{4}},\\ \\
\displaystyle{-\frac{1}{4}+O\left(\frac{1}{\mu_n}\right)},&\text{if } \displaystyle{\alpha^2\geq\frac{1}{4}},
\end{array}
\right.\ \ \ \text{and}\ \ \ 
\left|\text{Im}\left\{\lambda_{k,n}^{\pm}\right\}\right|\geq\mu_n.
\end{equation*}
In addition to that, from Proposition \ref{Coupled Theorem 2.7}, the system of eigenvectors and root vectors of $\mathcal{A}$ forms a Riesz basis of $\mathcal{H}$. Then,  applying  Proposition \ref{Chapter pr-43}, we get that System \eqref{Coupled Equation 2.1} is exponentially stable. Thus, the proof is  complete.
\xqed{$\square$}
\section{Polynomial stability}\label{Coupled Section 4}
\noindent In this Section,  we consider the remaining cases when $a=1$ and $\gamma<0$ or when $a\neq1$ and $\gamma\leq0$. In these cases, we prove that System  \eqref{Coupled Equation 2.1} is polynomially stable. More precisely, we find the optimal polynomial decay rate.  Our main result in this Section is the following theorem.
\begin{thm}\label{Coupled Theorem 2.8}
There exists a positive constant $C>0$ such that for every $u_0\in D\left(\mathcal{A}\right)$, the energy of System \eqref{Coupled Equation 2.1} has the  polynomial decay rate
\begin{equation}\label{Coupled Equation 2.42}
E\left(t\right)
\leq 
\frac{C}{t^{\delta\left(\gamma\right)}}\left\|\mathcal{A} u_0\right\|^2_{\mathcal{H}},
\quad \forall\ t>0,
\end{equation}
 where
\begin{equation*}
\delta\left(\gamma\right)=\left\{
\begin{array}{lll}

-\frac{1}{\gamma},&   \text{if }   a=1\text{ and } \gamma<0, 

                   \\ \\ 
 
\frac{1}{1-\gamma}, & \text{if }   a\neq1\text{ and } \gamma\leq0.

\end{array}
\right.
\end{equation*}
In addition, the energy decay rate in \eqref{Coupled Equation 2.42}  is optimal according to Definition \ref{Chapter pr-35}.
\end{thm}
\noindent For the  proof of  Theorem \ref{Coupled Theorem 2.8}, we need, as in the previous case, to study the asymptotic behavior  of the eigenvalues  $\lambda_{1,n}^{\pm},\  \lambda_{2,n}^{\pm}$ and  the  corresponding eigenvectors  $e_{1,n}^{\pm},\ e_{2,n}^{\pm}$.  We start with the case when $a=1$ and  $\gamma<0$.
\begin{lem}\label{Coupled Lemma 2.9}
{Suppose that $a=1$ and $\gamma<0$.  Let $N\in\mathbb{N}$ be the integer part of $\frac12-\gamma$, i.e., the unique integer such that 
             $$2N\leq1-2\gamma{\color{black}< 2N+2}.$$ 
Then the  eigenvalues $\lambda_{1,n}^{\pm},\ \lambda_{2,n}^{\pm}$
of System \eqref{Coupled Equation 2.1} satisfy the following asymptotic expansions
\begin{equation}\label{Coupled Equation 2.43}
\lambda_{1,n}^{\pm}=\left\{
\begin{array}{lll}

\displaystyle{\pm  i   \mu_n+\frac{  i   \alpha}{2}-\frac{1}{4 \mu_n^{-2\gamma}}+O\left(\frac{1}{\mu_n^{\min\left(1,-4\gamma\right)}}\right)}&\text{if }0<-2\gamma<1,

             \\ \\

\displaystyle{\pm  i   \mu_n+\frac{  i   \alpha}{2}\pm  i    \sum_{k=1}^{N}\frac{\alpha_{k}}{\mu_n^{2k-1}}-\frac{1}{4\mu_n^{-2\gamma}}+O\left(\frac{1}{\mu_n^{\min\left(2N+1,1-2\gamma\right)}}\right)}&\text{if }-2\gamma\geq1,

\end{array}
\right.
\end{equation}
\begin{equation}\label{Coupled Equation 2.44}
\lambda_{2,n}^{\pm}=\left\{
\begin{array}{lll}

\displaystyle{\pm  i   \mu_n-\frac{  i   \alpha}{2}-\frac{1}{4\mu_n^{-2\gamma}}+O\left(\frac{1}{\mu_n^{\min\left(1,-4\gamma\right)}}\right)}& \text{if }0<-2\gamma<1,

               \\ \\

\displaystyle{\pm  i   \mu_n-\frac{  i   \alpha}{2}\pm   i   \sum_{k=1}^{N}\frac{\alpha_{k}}{\mu_n^{2k-1}}-\frac{1}{4\mu_n^{-2\gamma}}+O\left(\frac{1}{\mu_n^{\min\left(2N+1,1-2\gamma\right)}}\right)}&  \text{if }-2\gamma\geq1,

\end{array}
\right.
\end{equation}
where $\alpha_1=\dfrac{\alpha^2}{8}$ and $\alpha_k=\dfrac{\left(2k-3\right)!\alpha^{2k}\left(-1\right)^{k-1}}{2^{4k-2}\left(k-2\right)! k!},\ \forall\ k\geq2$.}
\end{lem}
\begin{proof} 
Assume  that $a=1$ and $\gamma< 0$. For $n$ large enough, from \eqref{charact eq}, we get that 
\begin{equation}\label{Coupled Equation 2.15.33}
\lambda^2_{1,n}+\frac{\mu^{2\gamma}_n-i\sqrt{4\alpha^2-\mu^{4\gamma}_n}}{2}\lambda_{1,n}+\mu_n^2=0
\ \ \ \text{and}\ \ \ 
\lambda^2_{2,n}+\frac{\mu^{2\gamma}_n+i\sqrt{4\alpha^2-\mu^{4\gamma}_n}}{2}\lambda_{2,n}+\mu_n^2=0.
\end{equation}
Solving Equation \eqref{Coupled Equation 2.15.33}, we obtain
\begin{equation}\label{branch1-case a=1, gamma<0}
\begin{array}{ll}

\displaystyle{\lambda_{1,n}^\pm=\frac{1}{4}\left(-\mu_n^{2\gamma}\mp\sqrt{2
{{\sqrt{16\mu^4_n+4(2\alpha^2-\mu_n^{4\gamma})\mu^2_n +\alpha^4}-8\mu^2_n-2\alpha^2+\mu_n^{4\gamma}}}}\right)  }\\ \\ \hspace{3cm} \displaystyle{+\frac{i}{4}\left(\sqrt{4\alpha^2-\mu_n^{4\gamma}}\pm\sqrt{2
{{\sqrt{16\mu^4_n+4(2\alpha^2-\mu_n^{4\gamma})\mu^2_n +\alpha^4}+8\mu^2_n+2\alpha^2-\mu_n^{4\gamma}}}}\right)}

\end{array} 
\end{equation}
and
\begin{equation}\label{branch2-case a=1, gamma<0}
\begin{array}{ll}

\displaystyle{\lambda_{2,n}^\pm=\frac{1}{4}\left(-\mu_n^{2\gamma}\mp\sqrt{2
{{\sqrt{16\mu^4_n+4(2\alpha^2-\mu_n^{4\gamma})\mu^2_n +\alpha^4}-8\mu^2_n-2\alpha^2+\mu_n^{4\gamma}}}}\right)  }\\ \\ \hspace{3cm} \displaystyle{+\frac{i}{4}\left(-\sqrt{4\alpha^2-\mu_n^{4\gamma}}\pm\sqrt{2
{{\sqrt{16\mu^4_n+4(2\alpha^2-\mu_n^{4\gamma})\mu^2_n +\alpha^4}+8\mu^2_n+2\alpha^2-\mu_n^{4\gamma}}}}\right)}.

\end{array} 
\end{equation}
We have
\begin{equation}\label{kp1}
\sqrt{2\sqrt{16\mu^4_n+4(2\alpha^2-\mu_n^{4\gamma})\mu^2_n +\alpha^4}-8\mu^2_n-2\alpha^2+\mu_n^{4\gamma}}=O\left(\frac{1}{\mu_n^{1-2\gamma}}\right),
\end{equation}
\begin{equation}\label{kp2}
\sqrt{4\alpha^2-\mu_n^{4\gamma}}=2\alpha+O\left(\frac{1}{\mu_n^{-4\gamma}}\right)
\end{equation}
and
\begin{equation}\label{kp3}
\begin{array}{ll}

\displaystyle{\sqrt{\sqrt{16\mu^4_n+4(2\alpha^2-\mu_n^{4\gamma})\mu^2_n +\alpha^4}+8\mu^2_n+2\alpha^2-\mu_n^{4\gamma}}}\\ \\
\displaystyle{=2\sqrt{2}\mu_n\sqrt{\sqrt{\left(1+\frac{\alpha^2}{4\mu^2_n}\right)^2+O\left(\frac{1}{\mu_n^{2-4\gamma}}\right)}+1+\frac{\alpha^2}{4\mu^2_n}+O\left(\frac{1}{\mu_n^{2-4\gamma}}\right)}}\\ \\ \displaystyle{=4 \mu_n \sqrt{1+\frac{\alpha^2}{4\mu_n^2}+O\left(\frac{1}{\mu_n^{2-4\gamma}}\right)}}.

\end{array}
\end{equation}
We next proceed by studying two cases.
\\[0.1in] 
\noindent\textbf{Case 1.} If $0<-2\gamma<1$, 
 then $N=0$ and, from \eqref{kp3}, we obtain 
\begin{equation}\label{Coupled Equation 2.52}
4 \mu_n \sqrt{1+\frac{\alpha^2}{4\mu_n^2}+O\left(\frac{1}{\mu_n^{2-4\gamma}}\right)}=4 \mu_n \sqrt{1+O\left(\frac{1}{\mu_n^{2}}\right)}=
  4\mu_n+O\left(\frac{1}{\mu_n}\right).
\end{equation}
Inserting \eqref{kp1}, \eqref{kp2}, \eqref{kp3} and \eqref{Coupled Equation 2.52} in \eqref{branch1-case a=1, gamma<0} and \eqref{branch2-case a=1, gamma<0}, we get
\begin{equation}\label{Coupled Equation 2.53}
\lambda_{1,n}^\pm=\pm  i   \mu_n+\frac{  i   \alpha}{2}-\frac{1}{4\mu_n^{-2\gamma}}+O\left(\frac{1}{\mu_n^{\min\left(1,-4\gamma\right)}}\right),
\end{equation}
and
\begin{equation}\label{Coupled Equation 2.54}
\lambda_{2,n}^\pm=\pm  i   \mu_n-\frac{  i   \alpha}{2}-\frac{1}{4\mu_n^{-2\gamma}}+O\left(\frac{1}{\mu_n^{\min\left(1,-4\gamma\right)}}\right).
\end{equation}
\\[0.1in] \noindent\textbf{Case 2.} If $-2\gamma\geq 1$, then  from \eqref{Coupled Equation 2.53} and \eqref{Coupled Equation 2.54}, we need to increase the order of the finite expansion. Consider  the integer part $N\in\mathbb{N}$ of $\frac12-\gamma$, which is positive.
Setting  $x=\frac{\alpha^2}{4\mu_n^2}+O\left(\frac{1}{\mu_n^{2-4\gamma}}\right)$,  we have
\begin{equation}\label{Coupled Equation 2.55}
\left(1+x\right)^{\frac{1}{2}}=1+\sum_{k=1}^N{\beta_k}{x^k}+O\left(x^{N+1}\right),
\end{equation}
where
\begin{equation*}
\beta_1=\frac{1}{2}
\ \ \ \text{and}\ \ \ 
\beta_k=\frac{\left(2k-3\right)!\left(-1\right)^{k-1}}{2^{2k-2}\left(k-2\right)! k!},\ \forall\ k\geq2.
\end{equation*}
Therefore,
\begin{equation*}
\left(1+\frac{\alpha^2}{4\mu_n^2}+O\left(\mu^{4\gamma-2}_n\right)\right)^{\frac{1}{2}}
=
1+\sum_{k= 1}^N\beta_k\left(\frac{\alpha^2}{4\mu_n^{2}}\right)^k+O\left(\frac{1}{\mu^{2-4\gamma}_n}\right)+O\left(\frac{1}{\mu_n^{2N+2}}\right).
\end{equation*}
Taking $\alpha_k=\displaystyle{\frac{\alpha^{2k}\beta_k}{2^{2k}}}$ for every positive integer $k$ and noticing that $2N+2\leq 2-4\gamma$, we obtain
\begin{equation}\label{Coupled Equation 2.56}
\left(1+\frac{\alpha^2}{4\mu_n^2}+O\left(\mu^{4\gamma-2}_n\right)\right)^{\frac{1}{2}}
=
1+\sum_{k=1}^N\frac{\alpha_k}{\mu_n^{2k}}+O\left(\frac{1}{\mu_n^{2N+2}}\right).
\end{equation}
Inserting \eqref{Coupled Equation 2.56} in \eqref{kp3}, we get
\begin{equation}\label{kp3new}
\frac{1}{4}\sqrt{\sqrt{16\mu^4_n+4(2\alpha^2-\mu_n^{4\gamma})\mu^2_n +\alpha^4}+8\mu^2_n+2\alpha^2-\mu_n^{4\gamma}}=\mu_n+\sum_{k=1}^N\frac{\alpha_k}{\mu_n^{2k-1}}+O\left(\frac{1}{\mu_n^{2N+1}}\right).
\end{equation}
Substituting \eqref{kp1}, \eqref{kp2} and \eqref{kp3new} in \eqref{branch1-case a=1, gamma<0},  then using the condition $1-2\gamma\leq2N+2$, we obtain
\begin{equation}\label{Coupled Equation 2.59}
\lambda_{1,n}^\pm
=-\frac{1}{4\mu_n^{-2\gamma}}+  i   \left(\pm\mu_n+\frac{\alpha}{2}\pm \sum_{k= 1}^N\frac{\alpha_k}{\mu_n^{2k-1}}\right)+O\left(\frac{1}{\mu_n^{\min\left(2N+1,1-2\gamma\right)}}\right).
\end{equation}
 Then, from \eqref{Coupled Equation 2.53} and \eqref{Coupled Equation 2.59}, we get \eqref{Coupled Equation 2.43}. On the other hand, inserting \eqref{kp1}, \eqref{kp2} and \eqref{kp3new} in \eqref{Coupled Equation 2.15.33},  then using the condition $1-2\gamma\leq2N+2$, we get
 \begin{equation}\label{Coupled Equation 2.60}
\lambda_{2,n}^\pm
=-\frac{1}{4\mu_n^{-2\gamma}}+  i   \left(\pm\mu_n-\frac{\alpha}{2}\pm \sum_{k= 1}^N\frac{\alpha_k}{\mu_n^{2k-1}}\right)+O\left(\frac{1}{\mu_n^{\min\left(2N+1,1-2\gamma\right)}}\right).
\end{equation}
Finally, from \eqref{Coupled Equation 2.54} and \eqref{Coupled Equation 2.60}, we obtain \eqref{Coupled Equation 2.44}. Thus, the proof is  complete.
\end{proof}
\noindent Before we study the asymptotic behavior of the eigenvalues in case $a\neq 1$ and $\gamma\leq0$, we prove the following lemma. 
\noindent We now study the asymptotic behavior of the eigenvalues in the case when $a\neq 1$ and $\gamma\leq0$. We prove the following lemma.
\begin{lem}\label{Coupled Lemma 2.12}
{Suppose that $a\neq 1$ and $\gamma\leq0$. Let $N\in\mathbb{N}$ be the integer part of $\frac12-\gamma$. 
Then, we have
\begin{equation}\label{Coupled Equation 2.113}
\lambda_{1,n}^\pm=\left\{\begin{array}{lll}

\displaystyle{\pm  i   \sqrt{a}\mu_n-\frac{1}{2\mu_n^{-2\gamma}}+O\left(\frac{1}{\mu_n}\right)}&\text{if } 0\leq-2\gamma<1,

                 \\ \\

\displaystyle{\pm i\sqrt{a}\mu_n\pm i\sqrt{a}\sum_{\ell=1}^{N}\frac{C_1\left(\ell\right)}{\mu_n^{2\ell-1}}-\frac{1}{2\mu_n^{-2\gamma}}+O\left(\frac{1}{\mu_n^{2N+1}}\right)}& \text{if } -2\gamma\geq1

\end{array}
\right.
\end{equation}
and
\begin{equation}\label{Coupled Equation 2.114}
\lambda_{2,n}^\pm=\left\{\begin{array}{lll}

\displaystyle{\pm i\mu_n\mp\frac{i\alpha^2}{\left( a-1\right)\mu_n} -
\frac{\alpha^2}{2\left(a-1\right)^2\mu_n^{2-2\gamma}}+O\left(\frac{1}{\mu_n^{3}}\right)}&\text{if } 0\leq-2\gamma<1,

                     \\ \\
                     
\displaystyle{\pm i\mu_n\pm i\sum_{\ell=1}^{N+1}\frac{C_2\left(\ell\right)}{\mu_n^{2\ell-1}}-\frac{\alpha^2}{2\left(a-1\right)^2\mu_n^{2-2\gamma}}+O\left(\frac{1}{\mu_n^{2N+3}}\right)}& \text{if } -2\gamma\geq1,

\end{array}
\right.
\end{equation}
where, for $1\leq\ell\leq N+1$,   $C_1\left(\ell\right)$ and $C_2\left(\ell\right)$ are real numbers depending only on $a$ and $\alpha$. }
\end{lem}
\noindent For the proof of Lemma \ref{Coupled Lemma 2.12}, we need the following lemmas.
\begin{lem}\label{Coupled Lemma 2.10}
{Suppose that $a\neq 1$ and $\gamma\leq0$.  Let $N\in\mathbb{N}$ be the integer part of $\frac12-\gamma$. 
 Then the  eigenvalues $\lambda_{1,n}^{\pm},\ \lambda_{2,n}^{\pm}$ satisfy the following asymptotic expansion
\begin{equation}\label{Coupled Equation 2.67}
\left\{
\begin{array}{lll}

\displaystyle{\lambda_{1,n}^\pm=\pm  i   \sqrt{a}\mu_n-\frac{1}{2\mu_n^{-2\gamma}}+O\left(\frac{1}{\mu_n}\right)}&\text{if }0\leq-2\gamma<1,

               \\ \\

\displaystyle{\frac{\left(\lambda^\pm_{1,n}\right)^2}{a}+\mu_n^2+\sum_{k=1}^{N}\frac{D(k)}{\left(\lambda^{\pm}_{1,n}\right)^{2k-2}}+\frac{  i   D_1^\pm}{\mu_n^{-1-2\gamma}}+O\left(\frac{1}{\mu_n^{2N}}\right)=0}&\text{if } -2\gamma\geq1

\end{array}
\right.
\end{equation}
and
\begin{equation}\label{Coupled Equation 2.68}
\left\{\begin{array}{lll}

\displaystyle{\lambda^\pm_{2,n}=\pm i\mu_n\mp\frac{i\alpha^2}{2\left( a-1\right)\mu_n} -
\frac{\alpha^2}{2\left(a-1\right)^2\mu_n^{2-2\gamma}}+O\left(\frac{1}{\mu_n^{3}}\right)}&\text{if }0\leq-2\gamma<1,

             \\ \\

\displaystyle{\left(\lambda^\pm_{2,n}\right)^2+\mu_n^2- \sum_{k=1}^{N+1}
\frac{D\left(k\right)}{\left(\lambda^\pm_{2,n}\right)^{2k-2}}+\frac{  i   D^\pm_2}{\mu_n^{1-2\gamma}}+O\left(\frac{1}{\mu_n^{2N+2}}\right)=0}&\text{if } -2\gamma\geq1,

\end{array}
\right.
\end{equation}
where 
\begin{equation*}
D_1^\pm=\pm\frac{1}{\sqrt{a}},
\quad 
D^\pm_2=\pm\frac{\alpha^2}{\left(a-1\right)^2},
\quad 
D(1)=\frac{\alpha^2}{a-1}
\end{equation*}
and for every $k\geq2$, we have
\begin{equation*}
D\left(k\right)
=
\frac{2a^{k-1}\alpha^{2k}\left(2k-3\right)!}{\left(a-1\right)^{2k-1}\left(k-2\right)! k!}.
\end{equation*}}
\end{lem}
\begin{proof} First,  from \eqref{charact eq}, we have
\begin{equation}\label{Coupled Equation 2.69}
\dfrac{\left(a+1\right)\left(\lambda^\pm_{1,n}\right)^2+\lambda_{1,n}^\pm\mu_n^{2\gamma}-\left(a-1\right)\left(\lambda_{1,n}^\pm\right)^2P\left(\lambda_{1,n}^\pm,\mu_n\right)}{2a}+\mu_n^2=0
\end{equation}
and
\begin{equation}\label{Coupled Equation 2.70}
\dfrac{\left(a+1\right)\left(\lambda^\pm_{2,n}\right)^2+\lambda_{2,n}^\pm\mu_n^{2\gamma}+\left(a-1\right)\left(\lambda^\pm_{2,n}\right)^2P\left(\lambda_{2,n}^\pm,\mu_n\right)}{2a}+\mu_n^2=0,
\end{equation}
where
\begin{equation*}
P\left(\lambda,\mu_n\right)
=
\left(1-\frac{4a\alpha^2}{\left(a-1\right)^2\lambda^2}-\frac{2}{\left(a-1\right)\lambda\mu_n^{-2\gamma}}+\frac{1}{\left(a-1\right)^2\left(\lambda\mu_n^{-2\gamma}\right)^2}\right)^{\frac{1}{2}}.
\end{equation*}
The proof is based on finding the asymptotic expansion of \eqref{Coupled Equation 2.69} and \eqref{Coupled Equation 2.70}. We note that the order of expansion is chosen so that the term equivalent to $\Re\left(\lambda_{n}\right)$ appears.
\noindent  First, we  show \eqref{Coupled Equation 2.67}. Consequently, let $\lambda_n=\lambda_{1,n}^\pm$.  We consider two cases.
\\[0.1in]  \noindent\textbf{Case 1.} If $0\leq-2\gamma<1\left(\text{in this case }N=0\right),$  then from Lemma \ref{Coupled Lemma 2.1},  we get
\begin{equation}\label{Coupled Equation 2.71}
P\left(\lambda_n,\mu_n\right)
=
\left(1-\frac{2}{\left(a-1\right)\lambda_n\mu_n^{-2\gamma}}+O\left(\frac{1}{\mu^2_n}\right)\right)^{\frac{1}{2}}
=
1-\frac{1}{\left(a-1\right)\lambda_n\mu_n^{-2\gamma}}+O\left(\frac{1}{\mu^2_n}\right).
\end{equation}     
Inserting \eqref{Coupled Equation 2.71} in \eqref{Coupled Equation 2.69}, then using Lemma \ref{Coupled Lemma 2.1},  we obtain
\begin{equation}\label{Coupled Equation 2.72}
\frac{\left(\lambda^\pm_{1,n}\right)^2}{a}+\frac{\lambda_{1,n}^\pm}{a\mu_n^{-2\gamma}}+\mu_n^2+O(1)=0.
\end{equation}
Solving \eqref{Coupled Equation 2.72}, we obtain
\begin{equation}\label{Coupled Equation 2.73}
\lambda_{1,n}^\pm=\pm i\sqrt{a} \mu_n-\frac{1}{2\mu_n^{-2\gamma}}+O\left(\frac{1}{\mu_n}\right).
\end{equation}  
\textbf{Case 2.} If $-2\gamma\geq1$, then from \eqref{Coupled Equation 2.73}, the order of expansion increases and $N\geq1$. We have
\begin{equation}\label{Coupled Equation 2.74}
P\left(\lambda_n,\mu_n\right)=1+\sum_{k=1}^N{\beta_k}{x^k}+O\left(x^{N+1}\right),
\end{equation}
where we have set
\begin{equation}\label{x}
x=-\frac{4a\alpha^2}{\left(a-1\right)^2\lambda^2_n}-\frac{2}{\left(a-1\right)\lambda_n\mu_n^{-2\gamma}}+\frac{1}{\left(a-1\right)^2\left(\lambda_n\mu_n^{-2\gamma}\right)^2}.
\end{equation}
From Lemma \ref{Coupled Lemma 2.1}, we have
\begin{equation}\label{Coupled Equation 2.75}
x=-\frac{4a\alpha^2}{\left(a-1\right)^2\lambda^2_n}-\frac{2}{\left(a-1\right)\lambda_n\mu_n^{-2\gamma}}+O\left(\frac{1}{\mu_n^{2-4\gamma }}\right)
\end{equation}
and for every $ k\geq 2$ (in case $N\geq2$), we have
\begin{equation}\label{Coupled Equation 2.76}
x^k=\left(-\frac{4a\alpha^2}{\left(a-1\right)^2\lambda^2_n}\right)^k+O\left(\frac{1}{\mu_n^{2k-2\gamma-1}}\right).
\end{equation}
Inserting  \eqref{Coupled Equation 2.75} and \eqref{Coupled Equation 2.76} in \eqref{Coupled Equation 2.74} and using the condition $1-2\gamma<2N+2$, we get
\begin{equation}\label{Coupled Equation 2.77}
P\left(\lambda_n,\mu_n\right)=1+\sum_{k=1}^N\beta_k\left(-\frac{4a\alpha^2}{\left(a-1\right)^2\lambda^{2}_n}\right)^k-\frac{1}{\left(a-1\right)\lambda_n\mu_n^{-2\gamma}}+O\left(\frac{1}{\mu_n^{\min\left(2-4\gamma,3-2\gamma,2N+2\right)}}\right).
\end{equation}
Since  $N\geq1$ and $2N\leq1-2\gamma$, then 
\begin{equation}\label{Coupled Equation 2.78}
\min\left(2-4\gamma,3-2\gamma,2N+2\right)=2N+2.
\end{equation}
Therefore,
\begin{equation}\label{Coupled Equation 2.79}
P\left(\lambda_n,\mu_n\right)=1+\sum_{k=1}^N\beta_k\left(-\frac{4a\alpha^2}{\left(a-1\right)^2\lambda^{2}_n}\right)^k-\frac{1}{\left(a-1\right)\lambda_n\mu_n^{-2\gamma}}+O\left(\frac{1}{\mu_n^{2N+2}}\right).
\end{equation}
Substituting  \eqref{Coupled Equation 2.79} in \eqref{Coupled Equation 2.69}, then   using Lemma \ref{Coupled Lemma 2.1}, we obtain
\begin{equation}\label{Coupled Equation 2.80}
\frac{\left(\lambda^\pm_{1,n}\right)^2}{a}+\frac{\lambda_{1,n}^\pm}{a\mu_n^{-2\gamma}}+\mu_n^2+\sum_{k=1}^N\frac{D\left(k\right)}{\left(\lambda^\pm_{1,n}\right)^{2k-2}}+O\left(\frac{1}{\mu_n^{2N}}\right)=0.
\end{equation}
From   \eqref{Coupled Equation 2.80} and Lemma \ref{Coupled Lemma 2.1}, we obtain
\begin{equation}\label{Coupled Equation 2.81}
\frac{\left(\lambda^\pm_{1,n}\right)^2}{a}+\mu_n^2+O\left(1\right)=0.
\end{equation}
Consequently, we have
\begin{equation}\label{Coupled Equation 2.82}
\frac{\lambda_{1,n}^\pm}{a\mu_n^{-2\gamma}}=\pm\frac{  i   }{\sqrt{a}\mu_n^{-1-2\gamma}}+O\left(\frac{1}{\mu_n^{1-2\gamma}}\right).
\end{equation}
Inserting \eqref{Coupled Equation 2.82}   in  \eqref{Coupled Equation 2.80}, then using the condition $1-2\gamma\geq 2N$ and  Lemma \ref{Coupled Lemma 2.1}, we get 
\begin{equation}\label{Coupled Equation 2.83}
\frac{\left(\lambda^\pm_{1,n}\right)^2}{a}+\frac{i D_1^\pm}{\mu_n^{-1-2\gamma}}+\mu_n^2+\sum_{k=1}^N\frac{D\left(k\right)}{\left(\lambda^\pm_{1,n}\right)^{2k-2}}+O\left(\frac{1}{\mu_n^{2N}}\right)=0.
\end{equation}
 Finally, from  \eqref{Coupled Equation 2.73} and \eqref{Coupled Equation 2.83}, we get \eqref{Coupled Equation 2.67}. 
\\[0.1in]
 Our next aim is to prove \eqref{Coupled Equation 2.68} by similar computations.  Let $\lambda_n=\lambda_{2,n}^\pm$.  We have two cases.\\[0.1in]
\textbf{Case 1.} If $0\leq-2\gamma<1$, then
\begin{equation*}
P\left(\lambda_n,\mu_n\right)=1+\frac{x}{2}-\frac{x^2}{8}+\frac{x^3}{16}+O\left(x^{4}\right),
\end{equation*}
where $x$ is defined in \eqref{x}.
Using Lemma \ref{Coupled Lemma 2.1}, we get
\begin{equation}\label{Coupled Equation 2.84}
P\left(\lambda_n,\mu_n\right)=1-\frac{2a\alpha^2}{\left(a-1\right)^2\lambda^2_n}-\frac{1}{\left(a-1\right)\lambda\mu_n^{-2\gamma}}-\frac{2a\alpha^2}{\left(a-1\right)^3\lambda^3\mu_n^{-2\gamma}}+O\left(\frac{1}{\mu_n^{4}}\right).
\end{equation}
Inserting \eqref{Coupled Equation 2.84} in \eqref{Coupled Equation 2.70} and using Lemma \ref{Coupled Lemma 2.1}, we obtain
\begin{equation}\label{Coupled Equation 2.85}
\left(\lambda^\pm_{2,n}\right)^2+\mu_n^2-\frac{\alpha^2}{\left(a-1\right)^2\mu_n^{-2\gamma}\lambda_{2,n}^\pm}-\frac{\alpha^2}{a-1}+O\left(\frac{1}{\mu_n^{2}}\right)=0.
\end{equation}
From \eqref{Coupled Equation 2.85} and  Lemma \ref{Coupled Lemma 2.1}, we have
\begin{equation}\label{Coupled Equation 2.86}
\left(\lambda^\pm_{2,n}\right)^2=-\mu_n^2+O\left(1\right).
\end{equation}
Consequently, we have
\begin{equation}\label{Coupled Equation 2.87}
-\frac{\alpha^2}{\left(a-1\right)^2\mu_n^{-2\gamma}\lambda_{2,n}^\pm}   =   \pm 
\frac{i\alpha^2}{\left(a-1\right)^2\mu_n^{1-2\gamma}}
+O\left(\frac{1}{\mu_n^{3-2\gamma}}\right).
\end{equation}
Substituting \eqref{Coupled Equation 2.87} in \eqref{Coupled Equation 2.85}, we get
\begin{equation}\label{Coupled Equation 2.88}
\left(\lambda^\pm_{2,n}\right)^2+\mu_n^2-\frac{\alpha^2}{a-1} \pm 
\frac{i\alpha^2}{\left(a-1\right)^2\mu_n^{1-2\gamma}}+O\left(\frac{1}{\mu_n^{2}}\right)=0.
\end{equation}
Solving \eqref{Coupled Equation 2.88}, we obtain
\begin{equation}\label{Coupled Equation 2.89}
\lambda^\pm_{2,n}=\pm i\mu_n\mp\frac{i\alpha^2}{2\left( a-1\right)\mu_n} -
\frac{\alpha^2}{2\left(a-1\right)^2\mu_n^{2-2\gamma}}+O\left(\frac{1}{\mu_n^{3}}\right).
\end{equation}
\textbf{Case 2.} If $-2\gamma\geq 1,$ then the order of expansion equal to $N$ is not sufficient. Consequently, we let
\begin{equation}\label{Coupled Equation 2.90}
P\left(\lambda_n,\mu_n\right)=1+\sum_{k=1}^{N+1}{\beta_k}{x^k}+O\left(x^{N+2}\right),
\end{equation}
where $x$ is given in \eqref{x}. 
Using Lemma \ref{Coupled Lemma 2.1}, we obtain
\begin{equation}\label{Coupled Equation 2.91}
\beta_1 x+\beta_2 x^2=\sum_{k=1}^2\beta_k\left(-\frac{4a\alpha^2}{\left(a-1\right)^2\lambda^2_n}\right)^k-\frac{1}{\left(a-1\right)\lambda\mu_n^{-2\gamma}}-\frac{2a\alpha^2}{\left(a-1\right)^3\lambda^3\mu_n^{-2\gamma}}+O\left(\frac{1}{\mu_n^{4-4\gamma}}\right).
\end{equation}
If $N\geq2$, then for every $ k\geq 3$, we have
\begin{equation}\label{Coupled Equation 2.92}
x^k=\left(-\frac{4a\alpha^2}{\left(a-1\right)^2\lambda^2_n}\right)^k+O\left(\frac{1}{\mu_n^{2k-2\gamma-1}}\right).
\end{equation}
Inserting  \eqref{Coupled Equation 2.91} and \eqref{Coupled Equation 2.92} in \eqref{Coupled Equation 2.90}, we get
\begin{equation}\label{Coupled Equation 2.93}
P\left(\lambda_n,\mu_n\right)=1+\sum_{k=1}^{N+1}\beta_k\left(-\frac{4a\alpha^2}{\left(a-1\right)^2\lambda^{2}_n}\right)^k-\frac{1}{\left(a-1\right)\lambda_n\mu_n^{-2\gamma}}-\frac{2a\alpha^2}{\left(a-1\right)^3\lambda^3\mu_n^{-2\gamma}}+O\left(\frac{1}{\mu_n^{S(N,\gamma)}}\right),
\end{equation}
where $$S(N,\gamma)=\min\left(4-4\gamma,5-2\gamma,2N+4\right).$$
Since  $N\geq1$ and $2N\leq1-2\gamma$, we can easily check that
\begin{equation}\label{Coupled Equation 2.94}
S(N,\gamma)=2N+4.
\end{equation}
From \eqref{Coupled Equation 2.93} and \eqref{Coupled Equation 2.94}, we get
\begin{equation}\label{Coupled Equation 2.95}
P\left(\lambda_n,\mu_n\right)=1+\sum_{k=1}^{N+1}\beta_k\left(-\frac{4a\alpha^2}{\left(a-1\right)^2\lambda^{2}_n}\right)^k-\frac{1}{\left(a-1\right)\lambda_n\mu_n^{-2\gamma}}
-\frac{2a\alpha^2}{\left(a-1\right)^3\lambda^3\mu_n^{-2\gamma}}+O\left(\frac{1}{\mu_n^{2N+4}}\right).
\end{equation}
Inserting \eqref{Coupled Equation 2.95} in \eqref{Coupled Equation 2.70} and using Lemma \ref{Coupled Lemma 2.1}, we obtain
\begin{equation}\label{Coupled Equation 2.96}
\left(\lambda^\pm_{2,n}\right)^2+\mu_n^2-\frac{\alpha^2}{\left(a-1\right)^2\mu_n^{-2\gamma}\lambda_{2,n}^\pm}-\sum_{k=1}^{N+1}\frac{D(k)}{\left(\lambda^\pm_{2,n}\right)^{2k-2}}+O\left(\frac{1}{\mu_n^{2N+2}}\right)=0.
\end{equation}
From \eqref{Coupled Equation 2.96} and  Lemma \ref{Coupled Lemma 2.1}, we remark that
\begin{equation}\label{Coupled Equation 2.97}
\left(\lambda^\pm_{2,n}\right)^2=-\mu_n^2+O\left(1\right).
\end{equation}
Consequently, we obtain
\begin{equation}\label{Coupled Equation 2.98}
-\frac{\alpha^2}{\left(a-1\right)^2\mu_n^{-2\gamma}\lambda_{2,n}^\pm}   =   \pm 
\frac{i\alpha^2}{\left(a-1\right)^2\mu_n^{1-2\gamma}}
+O\left(\frac{1}{\mu_n^{3-2\gamma}}\right).
\end{equation}
Substituting \eqref{Coupled Equation 2.98} in 
\eqref{Coupled Equation 2.96}, we get
\begin{equation}\label{Coupled Equation 2.99}
\left(\lambda^\pm_{2,n}\right)^2+\mu_n^2- \sum_{k=1}^{N+1}
\frac{D\left(k\right)}{\left(\lambda^\pm_{2,n}\right)^{2k-2}}+\frac{  i   D^\pm_2}{\mu_n^{1-2\gamma}}+O\left(\frac{1}{\mu_n^{2N+2}}\right)=0.
\end{equation}
Finally, from \eqref{Coupled Equation 2.89} and \eqref{Coupled Equation 2.99}, we get \eqref{Coupled Equation 2.68}. Thus, the proof is  complete.
\end{proof}
\noindent Next, when $-2\gamma\geq1$, we try to replace the powers of $\lambda_n$ in \eqref{Coupled Equation 2.67} and in \eqref{Coupled Equation 2.68} with powers of $\mu_n$ as shown in the following lemma. 
\begin{lem}\label{Coupled Lemma 2.11}
{Suppose that $a\neq 1$ and  $\gamma\leq0$.  Let $N\in\mathbb{N}$ be the integer part of $\frac12-\gamma$. 
 If  $-2\gamma\geq1$, then the  eigenvalues $\lambda_{1,n}^{\pm},\ \lambda_{2,n}^{\pm}$ satisfy the following asymptotic expansion
\begin{equation}\label{Coupled Equation 2.100}
\frac{\left(\lambda^\pm_{1,n}\right)^2}{a}+\mu_n^2+\sum_{k=1}^{N}\frac{F_1(k)}{\mu_n^{2k-2}}+\frac{  i   D_1^\pm}{\mu_n^{-1-2\gamma}}+O\left(\frac{1}{\mu_n^{2N}}\right)=0
\end{equation}
and
\begin{equation}\label{Coupled Equation 2.101}
\left(\lambda^\pm_{2,n}\right)^2+\mu_n^2- \sum_{k=1}^{N+1}
\frac{F_2\left(k\right)}{\mu_n^{2k-2}}+\frac{  i   D^\pm_2}{\mu_n^{1-2\gamma}}+O\left(\frac{1}{\mu_n^{2N+2}}\right)=0,
\end{equation}
where, for $1\leq k\leq N+1$, the real numbers $F_1(k)$ and  $F_2\left(k\right)$  depend only on $a$ and $\alpha$. }
\end{lem}
\begin{proof}
We first show  \eqref{Coupled Equation 2.101} by proceeding in two steps. We start by proving the following asymptotic estimate
 \begin{equation}\label{Coupled Equation 2.102}
\left(\lambda^\pm_{2,n}\right)^2+\mu_n^2- \sum_{k=1}^{N}
\frac{F_2\left(k\right)}{\mu_n^{2k-2}}+O\left(\frac{1}{\mu_n^{2N}}\right)=0,
\end{equation}
where, for $1\leq k\leq N+1$, the numbers  $F_2\left(k\right)$ are real numbers depending only on $a$ and $\alpha$. For this aim, using \eqref{Coupled Equation 2.68}, Lemma \ref{Coupled Lemma 2.1}, and  the fact that $2N-2<2N\leq 1-2\gamma$, we get
\begin{equation}\label{Coupled Equation 2.103}
\left(\lambda^\pm_{2,n}\right)^2+\mu_n^2- \sum_{k=1}^{N}
\frac{D\left(k\right)}{\left(\lambda^\pm_{2,n}\right)^{2k-2}}+O\left(\frac{1}{\mu_n^{2N}}\right)=0.
\end{equation}
If $N\leq3$, then from \eqref{Coupled Equation 2.103} and Lemma \ref{Coupled Lemma 2.1}, we obtain
\begin{equation}\label{Coupled Equation 2.104}
\left(\lambda^\pm_{2,n}\right)^2+\mu_n^2-\frac{\alpha^2}{a-1}+O\left(\frac{1}{\mu_n^{2}}\right)=0.
\end{equation}
From \eqref{Coupled Equation 2.104}, we get
\begin{equation}\label{Coupled Equation 2.105}
\sum_{k=1}^{3}\frac{D(k)}{\left(\lambda^{\pm}_{2,n}\right)^{2k-2}}=\sum_{k=1}^{3}\frac{F_2(k)}{\mu_n^{2k-2}}+O\left(\frac{1}{\mu^{6}_{n}}\right),
\end{equation}
where 
\begin{equation*}
F_2\left(1\right)=\frac{\alpha^2}{a-1},\ F_2\left(2\right)=-\frac{a\alpha^4}{\left(a-1\right)^3}\ \text{ and } \
F_2\left(3\right)=\frac{a\left(a+1\right)\alpha^6}{\left(a-1\right)^5}.
\end{equation*}
 Inserting  \eqref{Coupled Equation 2.105} in  \eqref{Coupled Equation 2.103}, we obtain
\begin{equation}\label{Coupled Equation 2.106}
\left(\lambda^\pm_{2,n}\right)^2+\mu_n^2- \sum_{k=1}^{3}
\frac{F_2\left(k\right)}{\mu_n^{2k-2}}+O\left(\frac{1}{\mu_{n}^{6}}\right)=0.
\end{equation}
Therefore, if $N\leq3$, then from \eqref{Coupled Equation 2.106}, we get \eqref{Coupled Equation 2.102}. Next, if $3<N\leq5$, then from \eqref{Coupled Equation 2.106}, we obtain
\begin{equation}\label{Coupled Equation 2.107}
\sum_{k=1}^{5}\frac{D(k)}{\left(\lambda^{\pm}_{2,n}\right)^{2k-2}}=\sum_{k=1}^{5}\frac{F_2(k)}{\mu_n^{2k-2}}+O\left(\frac{1}{\mu^{10}_{n}}\right),
\end{equation}
where
\begin{equation*}
F_2\left(4\right)=-\frac{a\alpha^8\left(a^2+3a+1\right)}{\left(a-1\right)^7}
\ \ \text{ and }\ \
F_2\left(5\right)=\frac{a\alpha^{10}\left(a^2+5a+1\right)}{\left(a-1\right)^9}.
\end{equation*}
Substituting \eqref{Coupled Equation 2.107} in \eqref{Coupled Equation 2.103}, we get
\begin{equation}\label{Coupled Equation 2.108}
\left(\lambda^\pm_{2,n}\right)^2+\mu_n^2- \sum_{k=1}^{5}
\frac{F_2\left(k\right)}{\mu_n^{2k-2}}+O\left(\frac{1}{\mu_{n}^{10}}\right)=0.
\end{equation}
Therefore, if $3<N\leq5$, then from \eqref{Coupled Equation 2.108}, we get \eqref{Coupled Equation 2.102}. Similarly,  if $N>5$, we iterate the above the process in order to get  \eqref{Coupled Equation 2.102}.
\\[0.1in]
 Our next goal is to prove \eqref{Coupled Equation 2.101}. From \eqref{Coupled Equation 2.102}, we get 
\begin{equation*}
\frac{1}{\left(\lambda^\pm_{2,n}\right)^2}=-\frac{1}{\mu^2_n\left(1-\displaystyle{\sum_{k=1}^{N}}
\frac{F_2\left(k\right)}{\mu_n^{2k}}+O\left(\frac{1}{\mu_n^{2N+2}}\right)\right)}.
\end{equation*}
Therefore, we have
\begin{equation}\label{Coupled Equation 2.109}
\frac{1}{\left(\lambda^\pm_{2,n}\right)^2}=-\frac{1}{\mu_n^2}-\frac{1}{\mu^2_n}\displaystyle{\sum_{j=1}^{N-1}}\left(\sum_{k=1}^{N}
\frac{F_2\left(k\right)}{\mu_n^{2k}}\right)^{j}+O\left(\frac{1}{\mu_n^{2N+2}}\right).
\end{equation}
From \eqref{Coupled Equation 2.109} we can find a real number $F_2\left(N+1\right)$   depending on $a$ and $\alpha$ such that
\begin{equation}\label{Coupled Equation 2.110}
\sum_{k=1}^{N+1}\frac{D(k)}{\left(\lambda^{\pm}_{2,n}\right)^{2k-2}}=\sum_{k=1}^{N+1}\frac{F_2(k)}{\mu_n^{2k-2}}+O\left(\frac{1}{\mu_n^{2N+2}}\right).
\end{equation}
Inserting \eqref{Coupled Equation 2.110} in \eqref{Coupled Equation 2.68}, we get \eqref{Coupled Equation 2.101}.
\\[0.1in]
 \noindent Next, we show  \eqref{Coupled Equation 2.100}. If $N=1$, then   from \eqref{Coupled Equation 2.67},  we obtain \eqref{Coupled Equation 2.100}. Otherwise, if $N\geq2$, then similar to \eqref{Coupled Equation 2.102}, using \eqref{Coupled Equation 2.67}, Lemma \ref{Coupled Lemma 2.1}, and  the fact that $2N-4<2N-2\leq -1-2\gamma$, we get
\begin{equation}\label{Coupled Equation 2.111}
\frac{\left(\lambda^\pm_{1,n}\right)^2}{a}+\mu_n^2+ \sum_{k=1}^{N-1}
\frac{F_1\left(k\right)}{\mu_n^{2k-2}}+O\left(\frac{1}{\mu_n^{2N-2}}\right)=0,
\end{equation}
where for every $k=1,\ldots,N-1$ the numbers $F_1(k)$ are real numbers depending only on $a$ and $\alpha$. 
 Moreover,  similar to \eqref{Coupled Equation 2.110}, using \eqref{Coupled Equation 2.111}, it follows that
\begin{equation}\label{Coupled Equation 2.112}
\sum_{k=1}^{N}\frac{D(k)}{\left(\lambda^{\pm}_{1,n}\right)^{2k-2}}=\sum_{k=1}^{N}\frac{F_1(k)}{\mu_n^{2k-2}}+O\left(\frac{1}{\mu_n^{2N}}\right),
\end{equation}
where $F_1\left(N\right)$ is a real number  depending on $a$ and $\alpha.$ Substituting \eqref{Coupled Equation 2.112} in second estimation of \eqref{Coupled Equation 2.67}, we get \eqref{Coupled Equation 2.100}.
Thus, the proof is  complete.
\end{proof}
\noindent\textbf{Proof of Lemma \ref{Coupled Lemma 2.12}.} Suppose that $a\neq1$ and $\gamma\leq0$. First, if $0\leq-2\gamma<1,$ then the estimations \eqref{Coupled Equation 2.113} and \eqref{Coupled Equation 2.114}   are obtained from \eqref{Coupled Equation 2.67} and \eqref{Coupled Equation 2.68}. Otherwise, if  $-2\gamma\geq1$, then from  \eqref{Coupled Equation 2.100}, we get
\begin{equation}\label{Coupled Equation 2.117}
\lambda^\pm_{1,n}=\pm i\sqrt{a}\mu_n\left(1+x\right)^{\frac{1}{2}}, 
\end{equation}
where
\begin{equation}\label{Coupled Equation 2.118}
x=\sum_{k=1}^{N}\frac{F_1(k)}{\mu_n^{2k}}+\frac{  i   D_1^\pm}{\mu_n^{1-2\gamma}}+O\left(\frac{1}{\mu_n^{2N+2}}\right).
\end{equation}
Therefore, we have
\begin{equation}\label{Coupled Equation 2.119}
\left(1+x\right)^{\frac{1}{2}}=1+\sum_{j=1}^{N}{\beta_j}{x^j}+O\left(x^{N+1}\right),
\end{equation}
where, for every $j\geq2$, we have 
\begin{equation}\label{Coupled Equation 2.120}
x^j=\left(\sum_{k=1}^{N}\frac{F_1(k)}{\mu_n^{2k}}\right)^j+O\left(\frac{1}{\mu_n^{3-2\gamma}}\right)+O\left(\frac{1}{\mu_n^{2N+4}}\right).
\end{equation}
Since $1-2\gamma\geq 2N$, then for every $2\leq j\leq N$, we have 
\begin{equation}\label{Coupled Equation 2.121}
x^j=\left(\sum_{k=1}^{N}\frac{F_1(k)}{\mu_n^{2k}}\right)^j+O\left(\frac{1}{\mu_n^{2N+2}}\right)
\end{equation}
and
\begin{equation}\label{Coupled Equation 2.122}
x^{N+1}=\frac{1}{\mu^{2N+2}_n}\left(\sum_{k=1}^{N}\frac{F_1(k)}{\mu_n^{2k-2}}\right)^{N+1}+O\left(\frac{1}{\mu_n^{2N+4}}\right)=O\left(\frac{1}{\mu_n^{2N+2}}\right).
\end{equation}
Inserting  \eqref{Coupled Equation 2.121} and \eqref{Coupled Equation 2.122} in \eqref{Coupled Equation 2.119}, we obtain
\begin{equation}\label{Coupled Equation 2.124}
\left(1+x\right)^{\frac{1}{2}}=1+\frac{  i   D_1^\pm}{2\mu_n^{1-2\gamma}}+\sum_{j=1}^N\beta_j\left(\sum_{k=1}^{N}\frac{F_1(k)}{\mu_n^{2k}}\right)^j+O\left(\frac{1}{\mu_n^{2N+2}}\right).
\end{equation}
On the other hand, for every $1\leq j\leq N$,  $1\leq k\leq N$, there exist  real numbers  $C_1\left(\ell\right),\ 1\leq \ell\leq N$,  depending only on $a$ and $\alpha$, such that
\begin{equation}\label{Coupled Equation 2.125}
\sum_{j=1}^N\beta_j\left(\sum_{k=1}^{N}\frac{F_1(k)}{\mu_n^{2k}}\right)^j=\sum_{\ell=1}^{N}\frac{C_1\left(\ell\right)}{\mu_n^{2\ell}}+O\left(\frac{1}{\mu_n^{2N+2}}\right).
\end{equation}
Substituting \eqref{Coupled Equation 2.125} in \eqref{Coupled Equation 2.124}, we obtain
\begin{equation}\label{Coupled Equation 2.126}
\left(1+x\right)^{\frac{1}{2}}=1+\frac{  i   D_1^\pm}{2\mu_n^{1-2\gamma}}+\sum_{\ell=1}^{N}\frac{C_1\left(\ell\right)}{\mu_n^{2\ell}}+O\left(\frac{1}{\mu_n^{2N+2}}\right).
\end{equation}
Inserting \eqref{Coupled Equation 2.126} in \eqref{Coupled Equation 2.117}, we get second estimation of \eqref{Coupled Equation 2.113}. Next, for $\lambda_{2,n}^\pm,$ from \eqref{Coupled Equation 2.101}, we obtain
\begin{equation}\label{Coupled Equation 2.127}
\lambda^\pm_{2,n}=\pm i\mu_n\left(1- \sum_{k=1}^{N+1}
\frac{F_2\left(k\right)}{\mu_n^{2k}}+\frac{  i   D^\pm_2}{\mu_n^{3-2\gamma}}+O\left(\frac{1}{\mu_n^{2N+4}}\right)\right)^{\frac{1}{2}}.
\end{equation}
Similar to \eqref{Coupled Equation 2.126}, we can show that
\begin{equation}\label{Coupled Equation 2.128}
\left(1-\sum_{k=1}^{N+1}\frac{F_2(k)}{\mu_n^{2k}}+\frac{  i   D_2^\pm}{\mu_n^{3-2\gamma}}+O\left(\frac{1}{\mu_n^{2N+4}}\right)\right)^{\frac{1}{2}}=1+\frac{  i   D_2^\pm}{2\mu_n^{3-2\gamma}}+\sum_{\ell=1}^{N+1}\frac{C_2\left(\ell\right)}{\mu_n^{2\ell}}+O\left(\frac{1}{\mu_n^{2N+4}}\right),
\end{equation}
where for every $ 1\leq \ell\leq N+1,$ the numbers $C_2\left(\ell\right)$ are real numbers depending only on $a$ and $\alpha$. Inserting \eqref{Coupled Equation 2.128} in \eqref{Coupled Equation 2.127}, we obtain  second estimation of \eqref{Coupled Equation 2.114}. Thus, the proof is complete.\xqed{$\square$}
\\[0.1in]\noindent We now study the asymptotic behavior of the eigenvectors in the different cases $a=1$ and  $\gamma<0$ or $a\neq1$ and $\gamma\leq0$. We prove the following lemma.
\begin{lem}\label{eigenvectorspoly}
{If $a=1$ and  $\gamma<0$ or $a\neq1$ and $\gamma\leq0$, and $N$ equal to the integer part of $\frac12-\gamma$, then
the eigenvectors $e_{1,n}^{\pm}$ and $e_{2,n}^{\pm}$ of System \eqref{Coupled Equation 2.1} satisfy the following asymptotic expansion\\[0.1in]
\noindent If $a=1$ and  $\gamma<0$, then
\begin{equation}\label{Coupled Equation 2.45}
e_{1,n}^{\pm}=\frac{1}{2}\left(\frac{e_n}{\pm  i   \mu_n},e_n, \frac{e_n}{\mp\mu_n}, -ie_n\right)^{\top}
+\left(O\left(\frac{1}{\mu_n^{2}}\right), 0, O\left(\frac{1}{\mu_n^{\min\left(2,1-2\gamma\right)}}\right),  O\left(\frac{1}{\mu_n^{-2\gamma}}\right)\right)^{\top}
\end{equation}
and
\begin{equation}\label{Coupled Equation 2.46}
e_{2,n}^{\pm}=\frac{1}{2}\left( \frac{e_n}{\mp\mu_n}, -  i    e_n,  \frac{e_n}{\pm  i   \mu_n}, e_n  \right)^{\top}+
\left(O\left(\frac{1}{\mu_n^{\min\left(2,1-2\gamma\right)}}\right),  O\left(\frac{1}{\mu_n^{-2\gamma}}\right),    O\left(\frac{1}{\mu_n^{2}}\right) ,0 \right)^{\top}.
\end{equation}
If $a\neq0$ and $\gamma\leq0$, then
\begin{equation}\label{Coupled Equation 2.115}
e_{1,n}^{\pm}=\frac{1}{\sqrt{2}}\left(\frac{e_n}{\pm  i   \sqrt{a}\mu_n}, e_n,0,0\right)^{\top}+\left(O\left(\frac{1}{\mu^2_n}\right),0, O\left(\frac{1}{\mu_n^2}\right),O\left(\frac{1}{\mu_n}\right) \right)^{\top}
\end{equation}
and
\begin{equation}\label{Coupled Equation 2.116}
e_{2,n}^{\pm}=\frac{1}{\sqrt{2}}\left(0,0,\frac{e_n}{\pm  i   \mu_n},e_n \right)^{\top}+\left(O\left(\frac{1}{\mu^2_n}\right),O\left(\frac{1}{\mu_n}\right),O\left(\frac{1}{\mu^2_n}\right),0\right)^{\top}.
\end{equation}}
\end{lem}
\begin{proof}
Let $\lambda_{1,n}^{\pm},\  \lambda_{2,n}^{\pm}$ be the solutions of \eqref{Coupled Equation 2.15}. Setting 
\begin{equation*}
B_{1,n}=\frac{B_{1,n}^\pm}{\lambda_{1,n}^\pm}\  
   \ \ \ \ \text{and} \ \ \   
 C_{2,n}=\frac{C_{2,n}^\pm}{\lambda_{2,n}^\pm}
\end{equation*}
 in \eqref{Coupled Equation 2.5}, we get
 \begin{equation*}
   C_{1,n}=\frac{\alpha}{\left(\lambda_{1,n}^{\pm}\right)^2+\mu^2_n } B_{1,n}^{\pm}
   \ \ \ \ \text{and} \ \ \   
 B_{2,n} =\frac{\left(\lambda_{2,n}^{\pm}\right)^2+\mu^2_n}{\alpha\left(\lambda_{2,n}^{\pm}\right)^2} C_{2,n}^{\pm}.
\end{equation*}
Therefore, from  \eqref{Coupled Equation 2.4}, we obtain
\begin{equation}\label{Coupled Equation 2.8new87}
\begin{array}{ll}
e_{1,n}^{\pm}=B_{1,n}^{\pm}\left(\dfrac{e_n}{\lambda_{1,n}^{\pm}}, e_n ,  \dfrac{\alpha e_n}{\left(\lambda_{1,n}^{\pm}\right)^2+\mu^2_n },  \dfrac{\alpha\lambda_{1,n}^{\pm}e_n}{\left(\lambda_{1,n}^{\pm}\right)^2+\mu^2_n }  \right)^{\top},

   \\ \\
   
e_{2,n}^{\pm}=C_{2,n}^{\pm}\left(\dfrac{\left(\lambda_{2,n}^{\pm}\right)^2+\mu^2_n}{\alpha\left(\lambda_{2,n}^{\pm}\right)^2}e_n, \dfrac{\left(\lambda_{2,n}^{\pm}\right)^2+\mu^2_n}{\alpha\lambda_{2,n}^{\pm}}e_n, \dfrac{e_n}{\lambda_{2,n}^{\pm}}, e_n    \right)^{\top},
\end{array}
\end{equation}
are the eigenvectors  corresponding  to the four eigenvalues $\lambda_{1,n}^{\pm},\  \lambda_{2,n}^{\pm}$, where  $B_{1,n}^{\pm},\  C_{2,n}^{\pm}\in\mathbb{C}$. Now, we prove \eqref{Coupled Equation 2.45} and \eqref{Coupled Equation 2.46}. From \eqref{Coupled Equation 2.43} and \eqref{Coupled Equation 2.44}, we have
\begin{equation}\label{Coupled Equation 2.61}
\lambda_{1,n}^{\pm}=\pm   i   \mu_n+O\left(1\right)\ \ \ \text{ and }\ \ \ 
\lambda_{2,n}^{\pm}=\pm   i   \mu_n+O\left(1\right).
\end{equation}
Therefore,
\begin{equation}\label{Coupled Equation 2.62}
\frac{1}{\lambda_{1,n}^{\pm}}=\frac{1}{\pm  i   \mu_n}+O\left(\frac{1}{\mu_n^2}\right)
\ \ \ \text{ and } \ \ \ 
\frac{1}{\lambda_{2,n}^{\pm}}=\frac{1}{\pm  i   \mu_n}+O\left(\frac{1}{\mu_n^2}\right).
\end{equation}
Next, from \eqref{Coupled Equation 2.15.33}, we obtain
\begin{equation}\label{Coupled Equation 2.62newnew}
\frac{\alpha}{\left(\lambda^{\pm}_{1,n}\right)^2+\mu_n^2}=-\frac{2\alpha}{
\left(\mu^{2\gamma}_n-i\sqrt{4\alpha^2-\mu^{4\gamma}_n}\right)\lambda_{1,n}^{\pm}}
\ \ \ \text{and}\ \ \ 
\frac{\left(\lambda^{\pm}_{2,n}\right)^2+\mu_n^2}{\alpha\lambda^{\pm}_{2,n}}=-\frac{\mu^{2\gamma}_n+i\sqrt{4\alpha^2-\mu^{4\gamma}_n}}{2\alpha}.
\end{equation}
From \eqref{Coupled Equation 2.62} and \eqref{Coupled Equation 2.62newnew}, we get
\begin{equation}\label{Coupled Equation 2.63}
\frac{\alpha}{\left(\lambda^{\pm}_{1,n}\right)^2+\mu_n^2}=\frac{1}{\mp\mu_n}+O\left(\frac{1}{\mu_n^{\min\left(2,1-2\gamma\right)}}\right),\ \frac{\alpha \lambda^{\pm}_{1,n}}{\left(\lambda^{\pm}_{1,n}\right)^2+\mu_n^2}=-  i   +O\left(\frac{1}{\mu_n^{-2\gamma}}\right)
\end{equation}
and
\begin{equation}\label{Coupled Equation 2.65}
\frac{\left(\lambda^{\pm}_{2,n}\right)^2+\mu_n^2}{\alpha\lambda^{\pm}_{2,n}}=-  i   +O\left(\frac{1}{\mu_n^{-2\gamma}}\right)
,\ \frac{\left(\lambda^{\pm}_{2,n}\right)^2+\mu_n^2}{\alpha\left(\lambda^{\pm}_{2,n}\right)^2}=\frac{1}{\mp\mu_n}+O\left(\frac{1}{\mu_n^{\min\left(2,1-2\gamma\right)}}\right).
\end{equation}
Setting $B_{1,n}^\pm=\frac{1}{2}$ in the first equation of \eqref{Coupled Equation 2.8new87}, then using \eqref{Coupled Equation 2.62} and \eqref{Coupled Equation 2.63} we get \eqref{Coupled Equation 2.45}. Finally, setting $C_{2,n}^\pm=\frac{1}{2}$ in the second equation of \eqref{Coupled Equation 2.8new87}, then using \eqref{Coupled Equation 2.62} and \eqref{Coupled Equation 2.65}, we obtain \eqref{Coupled Equation 2.46}. 
Our next aim is to prove \eqref{Coupled Equation 2.115} and \eqref{Coupled Equation 2.116}. From \eqref{Coupled Equation 2.113}, we have 
\begin{equation}\label{Coupled Equation 2.129}
\lambda_{1,n}^\pm=\pm  i   \sqrt{a}\mu_n+O\left(\frac{1}{\mu_n^{\min\left(1,-2\gamma\right)}}\right).
\end{equation}
Consequently, we obtain
\begin{equation}\label{Coupled Equation 2.130}
\frac{1}{\lambda^\pm_{1,n}}=\frac{1}{\pm  i   \sqrt{a}\mu_n}+O\left(\frac{1}{\mu^2_n}\right),
\quad 
\frac{\alpha}{\left(\lambda_{1,n}^\pm\right)^2+\mu_n^2}=O\left(\frac{1}{\mu^2_n}\right)
\ \text{and }\  
\frac{\alpha\lambda_{1,n}^\pm}{\left(\lambda_{1,n}^\pm\right)^2+\mu_n^2}=O\left(\frac{1}{\mu_n}\right).
\end{equation}
Setting $B^\pm_{1,n}=\frac{1}{\sqrt{2}}$ in the first estimation of  \eqref{Coupled Equation 2.8new87}, then using \eqref{Coupled Equation 2.130}, we get  \eqref{Coupled Equation 2.115}. On the  other hand, from \eqref{Coupled Equation 2.114}, we have 
\begin{equation}\label{Coupled Equation 2.131}
\lambda_{2,n}^\pm=\pm  i   \mu_n+O\left(\frac{1}{\mu_n}\right).
\end{equation}  
Consequently, we obtain
\begin{equation}\label{Coupled Equation 2.132}
\frac{1}{\lambda^\pm_{2,n}}=\frac{1}{\pm  i   \mu_n}+O\left(\frac{1}{\mu^2_n}\right),\quad \frac{\left(\lambda_{2,n}^\pm\right)^2+\mu_n^2}{\alpha\lambda_{2,n}^\pm}=O\left(\frac{1}{\mu_n}\right)\ 
\text{ and }\
 \frac{\left(\lambda_{2,n}^\pm\right)^2+\mu_n^2}{\alpha\left(\lambda_{2,n}^\pm\right)^2}=O\left(\frac{1}{\mu_n^2}\right).
\end{equation}
Finally,  setting $C^\pm_{2,n}=\frac{1}{\sqrt{2}}$  in the second estimation of  \eqref{Coupled Equation 2.8new87}, then using \eqref{Coupled Equation 2.132}, we obtain  \eqref{Coupled Equation 2.116}. Thus, the proof is complete.
\end{proof}
\noindent Similar to Corollary  \ref{Coupled Lemma 2.5}, we have the following Corollary. 
\begin{cor}\label{Coupled Lemma 2.13}
{ From Lemma \ref{eigenvectorspoly},  we deduce that
\begin{equation}\label{Coupled Equation 2.133}
\left(e_{1,n}^+,e_{1,n}^-,e_{2,n}^+,e_{2,n}^-\right)=\left(E_{1,n}^+,E_{1,n}^-,E_{2,n}^+,E_{2,n}^-\right)L_n,
\end{equation}
where
\begin{equation}\label{Coupled Equation 2.134}
L_n=\left\{
\begin{array}{lll}

\displaystyle{ \frac{1}{\sqrt{2}}
\begin{pmatrix}
1 &0&-  i   &0
           \\ \noalign{\medskip}
0 &1&0&-  i   
           \\ \noalign{\medskip}
-  i    &0&1&0
           \\ \noalign{\medskip}
0&-  i   &0 &1
           \\ \noalign{\medskip}
\end{pmatrix}+
O\left(\frac{1}{\mu_n^{\min\left(1,-2\gamma\right)}}\right),}& \displaystyle{\text{if } a=1\ \text{ and }\ \gamma<0,}

            \\ \\

\displaystyle{I_{4}+ O\left(\frac{1}{\mu_n}\right),}&\displaystyle{\text{if }  a\neq1\ \text{ and }\ \gamma\leq0,}
\end{array}
\right.
\end{equation}}
where $I_4$ denotes the identity matrix.\xqed{$\square$}
\end{cor}
\noindent Similar to Proposition \ref{Coupled Theorem 2.7}, we can prove the following proposition. 
\begin{prop}\label{Coupled Theorem 2.14new}
{Whether $a=1$ and  $\gamma<0$ or $a\neq1$ and $\gamma\leq0$, the
system of eigenvectors $\left\{e_{1,n}^+,e_{1,n}^-,e_{2,n}^+,e_{2,n}^-\right\}_{n\geq1}$ of $\mathcal{A}$ given in Lemma \ref{eigenvectorspoly} forms a Riesz basis in $\mathcal{H}$. In particular, all eigenvectors of $\mathcal{A}$ are of  the form  given in \eqref{Coupled Equation 2.4}.}\xqed{$\square$}
\end{prop}
\noindent \textbf{Proof of Theorem \ref{Coupled Theorem 2.8}.} First, if $a=1$ and $\gamma<0$, then from Lemma  \ref{Coupled Lemma 2.9}, we remark that $\Re\left(\lambda_n\right)+\displaystyle{\frac{1}{4\mu_{n}^{-2\gamma}}}=o(1)$, as $n$ tends to infinity. Therefore, by Proposition \ref{Chapter pr-45}, we get \eqref{Coupled Equation 2.42} where $\delta\left(\gamma\right)=\displaystyle{-\frac{1}{\gamma}}$. Next, if $a\neq1$ and $\gamma\leq0$, then from Lemma \ref{Coupled Lemma 2.12}, we remark that
 $$\Re\left(\lambda^\pm_{1,n}\right)\sim\displaystyle{-\frac{1}{2\mu_n^{-2\gamma}}} \ \ \textrm{and} \ \ \Re\left(\lambda^\pm_{2,n}\right)\sim\displaystyle{-
\frac{\alpha^2}{2\left(a-1\right)^2\mu_n^{2-2\gamma}}}.$$
Therefore, by Proposition \ref{Chapter pr-45}, we get \eqref{Coupled Equation 2.42} where $\delta\left(\gamma\right)=\displaystyle{\frac{1}{1-\gamma}}$.
Furthermore, from Proposition \ref{Coupled Theorem 2.14new}, the system of eigenvectors of $\mathcal{A}$ forms a Riesz basis in $\mathcal{H}$. Then, applying  Proposition \ref{Chapter pr-45}, we get the optimal polynomial energy decay rate given in \eqref{Coupled Equation 2.42}. Thus, the proof is  complete.\xqed{$\square$}

\section{Examples}\label{Coupled Section 5}
\noindent Let $\Omega\subset\mathbb{R}^N$ be a bounded open set with a smooth boundary $\Gamma$.\
\\[0.1in] \textbf{Example 1.} Consider the system of weakly coupled wave equations
\begin{equation}\label{Coupled Equation 3.1}
\left\{
\begin{array}{lll}

\utt-a\Delta u+\left(-\Delta\right)^{\gamma}\ut+\alpha \yt=0&\text{in }\Omega,

           \\ \noalign{\medskip}
           
\ytt-\Delta y-\alpha \ut=0&\text{in }\Omega,

           \\ \noalign{\medskip}

u=y=0&\text{on }\Gamma,

\end{array}
\right.
\end{equation}
with the following initial conditions
\begin{equation*}
u\left(x,0\right)=u_0\left(x\right)
,\ 
u_t\left(x,0\right)=u_1\left(x\right)
,\ 
y\left(x,0\right)=y_0\left(x\right)
,\ 
y_t\left(x,0\right)=y_1\left(x\right),
\end{equation*}
where $\gamma\leq0,\ a>0$, and $\alpha$ is a real number. We define the operator $A$ in $L^2\left(\Omega\right)$ by
\begin{equation*}
A=-\Delta \ \text{   with   }\ D\left(A\right)=H^2\left(\Omega\right)\cap H^1_0\left(\Omega\right).
\end{equation*}
We easily get that $A$ is a densely defined, closed, self-adjoint and coercive  operator with compact resolvent in $L^2\left(\Omega\right)$. We also assume that the spectrum of $A$ is simple. Note that this assumption is generic (in the Baire sense) with respect to the domain
$\Omega$, according to \cite{Dan01}. 
Therefore:\\[0.1in]
\noindent When $a=1$ and $\gamma=0$, applying Theorem \ref{Coupled Theorem},   we obtain an exponential energy decay rate  given by
\begin{equation*}
\left\|e^{t\mathcal{A}}u_0\right\|_{\mathcal{H}}
\leq
 M e^{-\epsilon t}\left\|u_0\right\|_{\mathcal{H}}, \quad t>0,\quad u_0\in {\mathcal{H}}.
\end{equation*}
When $a=1$ and $\gamma<0$ or when $a\neq0$ and $\gamma\leq0$, applying Theorem \ref{Coupled Theorem 2.8}, we obtain an optimal polynomial energy decay rate of the form
\begin{equation*}
E\left(t\right)\leq \frac{C}{t^{\delta\left(\gamma\right)}}\left\|u_0\right\|^2_{D\left(\mathcal{A}\right)},\quad  t>0,\quad u_0\in D\left(\mathcal{A}\right).
\end{equation*}
    %
             
%
$\newline$ 
\textbf{Example 2.} Consider the system of weakly coupled plate equations given by
\begin{equation}\label{Coupled Equation 3.2}
\left\{
\begin{array}{lll}

\utt+a\Delta^2 u+\left(\Delta^2\right)^{\gamma}\ut+\alpha \yt=0&\text{in }\Omega,

           \\ \noalign{\medskip}

\ytt+\Delta^2 y-\alpha \ut=0&\text{in }\Omega,

           \\ \noalign{\medskip}

u=\frac{\partial u}{\partial n}=y=\frac{\partial y}{\partial n}=0&\text{on }\Gamma,

\end{array}
\right.
\end{equation}
with the following initial conditions
\begin{equation*}
u\left(x,0\right)=u_0\left(x\right)
,\ 
u_t\left(x,0\right)=u_1\left(x\right)
,\ 
y\left(x,0\right)=y_0\left(x\right)
,\ 
y_t\left(x,0\right)=y_1\left(x\right),
\end{equation*}
where $\gamma\leq0,\ a>0$ and $\alpha$ is a real number. We define the operator $A$ in $L^2\left(\Omega\right)$ by
\begin{equation*}
A=\Delta^2 \ \text{   with   }\ D\left(A\right)=H^4\left(\Omega\right)\cap H^2_0\left(\Omega\right).
\end{equation*}
Here, $A$ is a densely defined, closed, self-adjoint and coercive operator with compact resolvent in $L^2\left(\Omega\right)$ and we furthermore assume that the spectrum of $A$ is simple, assumption which holds generically with respect to the domain $\Omega$, according to \cite{Dan01}. Then:\\[0.1in]
\noindent Applying Theorem \ref{Coupled Theorem}  with $a=1$ and $\gamma=0$, we get
\begin{equation*}
\left\|e^{t\mathcal{A}}u_0\right\|_{\mathcal{H}}
\leq
 M e^{-\epsilon t}\left\|u_0\right\|_{\mathcal{H}}, \quad t>0,\quad u_0\in {\mathcal{H}}.
\end{equation*}
Applying Theorem \ref{Coupled Theorem 2.8}, we obtain
\begin{equation*}
E\left(t\right)\leq \frac{C}{t^{\delta\left(\gamma\right)}}\left\|u_0\right\|^2_{D\left(\mathcal{A}\right)},\quad t>0,\quad u_0\in D\left(\mathcal{A}\right).
\end{equation*}
    %
             
%

\end{document}